\documentclass[3p]{elsarticle}
\usepackage{bm}
\usepackage[parfill]{parskip}
\usepackage[utf8]{inputenc}
\usepackage{url}
\usepackage{amsmath,amssymb,amsfonts,amsthm}
\usepackage{color}
\usepackage{bm}
\usepackage{graphicx}
\usepackage{subcaption}
\usepackage{lipsum}
\usepackage{amsfonts}
\usepackage{graphicx}
\usepackage{epstopdf}
\usepackage{array}
\usepackage[hidelinks]{hyperref}
\usepackage{lineno}
\usepackage{bbm}
\usepackage{xfrac}
\usepackage{varwidth}
\usepackage[nice]{nicefrac}

\usepackage{amsmath}
\usepackage{algorithm}
\usepackage[noend]{algpseudocode}

\ifpdf
  \DeclareGraphicsExtensions{.eps,.pdf,.png,.jpg}
\else
  \DeclareGraphicsExtensions{.eps}
\fi

\newtheorem{theorem}{Theorem}

\newtheorem{remark}{Remark}

\DeclareMathOperator*{\inter}{{\nicefrac{1}{2}}}

\makeatletter
\newsavebox\myboxA
\newsavebox\myboxB
\newlength\mylenA

\newcommand*\xoverline[2][0.75]{%
    \sbox{\myboxA}{$\m@th#2$}%
    \setbox\myboxB\null
    \ht\myboxB=\ht\myboxA%
    \dp\myboxB=\dp\myboxA%
    \wd\myboxB=#1\wd\myboxA
    \sbox\myboxB{$\m@th\overline{\copy\myboxB}$}
    \setlength\mylenA{\the\wd\myboxA}
    \addtolength\mylenA{-\the\wd\myboxB}%
    \ifdim\wd\myboxB<\wd\myboxA%
       \rlap{\hskip 0.5\mylenA\usebox\myboxB}{\usebox\myboxA}%
    \else
        \hskip -0.5\mylenA\rlap{\usebox\myboxA}{\hskip 0.5\mylenA\usebox\myboxB}%
    \fi}
\makeatother

\usepackage{tikz}
\usetikzlibrary{chains, positioning, arrows.meta, bending, shapes.arrows}
\usetikzlibrary{positioning,shapes,arrows,calc,shapes.geometric,backgrounds,fit}
\usepackage{pifont}
\usetikzlibrary{positioning,shapes,arrows,calc,shapes.geometric,backgrounds,fit}

\tikzstyle{block} = [rectangle,draw,minimum width=2em,align=center,rounded corners, minimum height=2em,scale=1.0]
\tikzstyle{blockleft} = [rectangle,draw,minimum width=2em,align=left,rounded corners, minimum height=2em,scale=1.0]
\tikzstyle{bigblock} = [rectangle,draw,minimum width=8em,align=center,rounded corners, minimum height=4em,scale=1.0]
\tikzstyle{connect} = [draw,-latex']
\tikzstyle{decision} = [diamond, draw, 
    text width=4.5em, text badly centered, node distance=3cm, inner sep=0pt]
\tikzstyle{line} = [draw, -latex']
\tikzstyle{cloud} = [draw, ellipse,fill=red!20, node distance=3cm,
    minimum height=2em]
\tikzstyle{linenoarrow}=[draw]

\makeatletter
\let\save@mathaccent\mathaccent
\newcommand*\if@single[3]{%
  \setbox0\hbox{${\mathaccent"0362{#1}}^H$}%
  \setbox2\hbox{${\mathaccent"0362{\kern0pt#1}}^H$}%
  \ifdim\ht0=\ht2 #3\else #2\fi
  }
\newcommand*\rel@kern[1]{\kern#1\dimexpr\macc@kerna}
\newcommand*\widebar[1]{\@ifnextchar^{{\wide@bar{#1}{0}}}{\wide@bar{#1}{1}}}
\newcommand*\wide@bar[2]{\if@single{#1}{\wide@bar@{#1}{#2}{1}}{\wide@bar@{#1}{#2}{2}}}
\newcommand*\wide@bar@[3]{%
  \begingroup
  \def\mathaccent##1##2{%
    \let\mathaccent\save@mathaccent
    \if#32 \let\macc@nucleus\first@char \fi
    \setbox\z@\hbox{$\macc@style{\macc@nucleus}_{}$}%
    \setbox\tw@\hbox{$\macc@style{\macc@nucleus}{}_{}$}%
    \dimen@\wd\tw@
    \advance\dimen@-\wd\z@
    \divide\dimen@ 3
    \@tempdima\wd\tw@
    \advance\@tempdima-\scriptspace
    \divide\@tempdima 10
    \advance\dimen@-\@tempdima
    \ifdim\dimen@>\z@ \dimen@0pt\fi
    \rel@kern{0.6}\kern-\dimen@
    \if#31
      \overline{\rel@kern{-0.6}\kern\dimen@\macc@nucleus\rel@kern{0.4}\kern\dimen@}%
      \advance\dimen@0.4\dimexpr\macc@kerna
      \let\final@kern#2%
      \ifdim\dimen@<\z@ \let\final@kern1\fi
      \if\final@kern1 \kern-\dimen@\fi
    \else
      \overline{\rel@kern{-0.6}\kern\dimen@#1}%
    \fi
  }%
  \macc@depth\@ne
  \let\math@bgroup\@empty \let\math@egroup\macc@set@skewchar
  \mathsurround\z@ \frozen@everymath{\mathgroup\macc@group\relax}%
  \macc@set@skewchar\relax
  \let\mathaccentV\macc@nested@a
  \if#31
    \macc@nested@a\relax111{#1}%
  \else
    \def\gobble@till@marker##1\endmarker{}%
    \futurelet\first@char\gobble@till@marker#1\endmarker
    \ifcat\noexpand\first@char A\else
      \def\first@char{}%
    \fi
    \macc@nested@a\relax111{\first@char}%
  \fi
  \endgroup
}
\makeatother

\journal{arXiv.org}

\newcommand{\TheTitle}{A robust collision source method for rank adaptive dynamical low-rank approximation in radiation therapy} 

\date{\today}

\usepackage{amsopn}

\newcommand{\RomanNumeralCaps}[1]
    {\MakeUppercase{\romannumeral #1}}

\journal{arXiv}

\begin{document}
\begin{frontmatter}

\title{\TheTitle}

\author[adressJonas]{Jonas Kusch}
\author[adressPia1,adressPia2,adressPia3]{Pia Stammer}

\address[adressJonas]{Karlsruhe Institute of Technology, Mathematical Science and Computational Methods, Karlsruhe, Germany, jonas.kusch@kit.edu}
\address[adressPia1]{Karlsruhe Institute of Technology, Steinbuch Centre for Computing,
Karlsruhe, Germany, pia.stammer@kit.edu}
\address[adressPia2]{German Cancer Research Center - DKFZ, Department of Medical Physics in
Radiation Oncology, Heidelberg, Germany}
\address[adressPia3]{HIDSS4Health - Helmholtz Information and Data Science School for Health,
Karlsruhe/Heidelberg, Germany}

\begin{abstract}
Deterministic models for radiation transport describe the density of radiation particles moving through a background material. In radiation therapy applications, the phase space of this density is composed of energy, spatial position and direction of flight. The resulting six-dimensional phase space prohibits fine numerical discretizations, which are essential for the construction of accurate and reliable treatment plans.
In this work, we tackle the high dimensional phase space through a dynamical low-rank approximation of the particle density. Dynamical low-rank approximation (DLRA) evolves the solution on a low-rank manifold in time. Interpreting the energy variable as a pseudo-time lets us employ the DLRA framework to represent the solution of the radiation transport equation on a low-rank manifold for every energy. Stiff scattering terms are treated through an efficient implicit energy discretization and a rank adaptive integrator is chosen to dynamically adapt the rank in energy. To facilitate the use of boundary conditions and reduce the overall rank, the radiation transport equation is split into collided and uncollided particles through a collision source method. Uncollided particles are described by a directed quadrature set guaranteeing low computational costs, whereas collided particles are represented by a low-rank solution. It can be shown that the presented method is L$^2$-stable under a time step restriction which does not depend on stiff scattering terms. Moreover, the implicit treatment of scattering does not require numerical inversions of matrices. Numerical results for radiation therapy configurations as well as the line source benchmark underline the efficiency of the proposed method. 
\end{abstract}


\end{frontmatter}

\section{Introduction}

Radiation therapy is one of the main tools in cancer treatment. In treatment planning, the set-up of radiation beams is optimized such that the tumor receives the prescribed dose, while minimizing the damage of surrounding risk organs or healthy tissue. Here, exact and fast dose calculation is of basic importance. Exact dose calculations however require the solution of a high dimensional coupled system of linear transport equations. Therefore, clinical dose calculation algorithms often rely on simplified \textit{pencil beam} models~\cite{ahnesjo1999dose}, which are based on the Fermi-Eyges theory of radiative transfer \cite{eyges1948multiple}. Although these models are computationally efficient, they can only describe layered heterogeneities \cite{hogstrom1981electron}. Thus simulation results are inaccurate, especially in cases including air cavities or other inhomogeneities \cite{hogstrom1981electron, krieger2005montecarlo}. 

More exact dose calculation can be achieved by an appropriate Monte Carlo (MC) algorithm, where individual interacting particles are directly simulated \cite{andreo1991montecarlo}. However, while recent performance-tuned MC implementations have achieved large run time improvements, the high computational costs and associated statistical noise still limit their clinical usage \cite{fippel2004monte, jia2012gpu}. 

The application of deterministic Boltzmann equations for dose calculation can achieve similar accuracy as MC simulations, but also exhibit the same computational complexity for grid-based solutions \cite{borgers1998complexity}.

In this work, we tackle the challenges arising from high-dimensional phase spaces in radiation therapy applications through dynamical low-rank approximation (DLRA) \cite{KochLubich07}. DLRA represents the solution by a low-rank ansatz. When the solution is an $n$ by $m$ time-dependent matrix with huge $n$ and $m$, such a low-rank representation can be defined by a singular value decomposition (SVD) truncated to a (small) rank $r$.\footnote{Note that the low-rank solution does not require a diagonal $r$ by $r$ singular value matrix. In fact, DLRA uses dense coefficient matrices.} Time evolution equations for each of the low-rank factors of the SVD are then derived by minimizing the residual while maintaining the solution's low-rank structure. When the original method requires $O(n\cdot m)$ operations per time step, the updates of the low-rank factors only require $O(r^2\cdot(n + m))$ operations. Robust integrators for the time evolution of these factors are the matrix projector--splitting integrator (PSI)~\cite{LubichOseledets} and the recently developed ``unconventional'' basis-update \& Galerkin low-rank matrix integrator of \cite{CeL21}. Unlike the unconventional integrator, the projector--splitting integrator can be extended to high order. However, since it includes a backward step in time, the projector--splitting integrator can yield instabilities for parabolic equations. Moreover, the backward time step can result in unstable schemes for hyperbolic problems, which is not an issue for the unconventional integrator \cite{kusch2021stability}. Following \cite{kusch2021stability}, stable discretizations of the PSI can however be constructed by using the continuous DLRA formulation proposed in \cite{EiL18}.

The efficiency of dynamical low-rank approximation has been demonstrated in several applications, including kinetic theory \cite{EiL18,einkemmer2019quasi,PeMF20,PeM20,einkemmer2020low,EiHY21,EiJ21,ceruti2021rank,kusch2021stability,ding2019dynamical}. Two main challenges of DLRA in the context of kinetic theory and radiation transport specifically are the preservation of mass as well as capturing the asymptotic limit. Methods to guarantee mass conservation include re-scaling strategies \cite{PeMF20}, a high-order low-order (HOLO) decomposition \cite{PeM20} and the incorporation of certain basis functions in the tangent space of low-rank functions \cite{EiJ21}. An asymptotic preserving scheme for dynamical low-rank approximation has been proposed in \cite{ding2019dynamical}. The key ingredients of this scheme are the ordering of low-rank updates and the choice of an implicit time discretization. A further method uses a HOLO scheme to guarantee preserving the asymptotic limit \cite{EiHY21}. Even though radiation therapy applications do not exhibit sufficiently strong scattering to fall into the diffusive regime, stiff scattering terms remain a challenge. While implicit time discretizations guarantee a stable treatment of such terms, their significantly increased computational costs pose difficulties. 

The efficiency of DLRA highly depends on the rank required to capture important solution characteristics. Choosing this rank sufficiently high to guarantee a satisfactory solution quality while at best maintaining low computational costs requires a great amount of intuition. Furthermore, a fixed choice of the rank does not capture the time evolution of the solution complexity. Rank adaptive DLRA integrators have for example been proposed in \cite{DeRV20,yang2020time,ceruti2021rank}.

This work presents a dynamical low-rank approximation for radiation therapy. To employ the DLRA framework in this setting, we formulate the energy dependency of the continuous slowing down equation as a pseudo-time. Dynamical low-rank approximation is then used to update the low-rank factors of the solution in energy. An efficient choice of the rank for every energy is provided through the rank adaptive integrator of \cite{ceruti2021rank}.
Further novelties of this work are:
\begin{enumerate}
\item \textit{A stable and efficient time discretization of stiff scattering terms.} Following \cite{ostermann2018convergence,kusch2021stability}, stiff scattering terms are split from the radiation transport equation. The unconventional integrator is used to time update the streaming part and the matrix projector-splitting integrator updates the scattering part. According to \cite{kusch2021stability}, the scattering part only requires an update of a single low-rank factor. By using an implicit time discretization on this part of the integrator and explicit updates on the remainder, we significantly reduce computational costs while allowing for a less restrictive CFL condition.
\item \textit{A first-collision source method to reduce the rank and impose boundary conditions.} Boundary conditions in radiation therapy are often Dirichlet conditions of uncollided particles traveling into a single direction. To efficiently incorporate this information into our solution ansatz, we perform a collided-uncollided split, see e.g. \cite{adams2002fast,hauck2013collision}. Here, collided and uncollided particles are treated in two separate equations. For the uncollided particles, an S$_N$ method with a directed quadrature set resolving only the small number of relevant directions can be used. The collided part of the solution is represented through dynamical low-rank approximation. This strips away highly peaked particle distributions in the low-rank approximation, thereby potentially reducing the rank. Furthermore, since the density of uncollided particles is zero at the boundary, imposing boundary conditions becomes straightforward.
\item \textit{A multilevel dynamical low-rank approximation.} The collided-uncollided split can be extended to $L$-collided splits. This can be interpreted as writing the solution as a telescoping sum. Expecting a reduced rank in every of these telescoping updates, individual dynamical low-rank approximations are derived for every term.
\end{enumerate}

This paper is structured as follows: Sections~\ref{sec:backgrounds} and \ref{sec:backgroundsDLRA} give an overview on methods used in this work to point to existing work and to fix notation. A general background to the used radiation transport methods, especially the continuous slowing down equation and numerical methods to solve it is given in Section~\ref{sec:backgrounds}. Section~\ref{sec:backgroundsDLRA} provides a recap of dynamical low-rank approximation as well as robust integrators for the DLRA evolution equations. The main method of this work is presented in Section~\ref{sec:mainMethod}. In Section~\ref{sec:L2Stability}, we derive a CFL condition which ensures stability and extend the results to the rank adaptive unconventional integrator in Section~\ref{sec:rankAdapt}.

\section{Recap: Mesoscopic transport models in radiation therapy}\label{sec:backgrounds}
\subsection{Continuous slowing down equation}
In the following, we discuss kinetic transport in the field of radiation treatment planning as well as related numerical methods used in this work. The task of computational mathematics in radiation therapy is to predict the transport of radiation particles in cancer patients. An accurate model is provided by the linear Boltzmann equation, which describes the dynamics of the particle density $\psi$ on a mesoscopic level \cite{frank2007fast}. The continuous slowing down (CSD) approximation \cite{larsen1997electron} to the linear Boltzmann equation reads
\begin{subequations}\label{eq:CSD1}
\begin{align}
    \mathbf{\Omega}\cdot\nabla_{\mathbf{x}}\psi(E,\mathbf{x},\mathbf{\Omega})+\rho(\mathbf{x})\Sigma_t(E)\psi(E,\mathbf{x},\mathbf{\Omega}) = &\int_{\mathbb{S}^2}\rho(\mathbf{x})\Sigma_s(E,\mathbf{\Omega}\cdot\mathbf{\Omega}')\psi(E,\mathbf{x},\mathbf{\Omega}')\,d\mathbf{\Omega}'\nonumber\\
    &+\partial_E\left(\rho(\mathbf{x})S_t(E)\psi(E,\mathbf{x},\mathbf{\Omega})\right)\;,\\
     \psi(E,\mathbf{x},\mathbf{\Omega}) &= \psi_{\text{BC}}(E,\mathbf{x},\mathbf{\Omega}) \quad \text{ for } \mathbf{x}\in\partial D\;,\\
     \psi(E_{\text{max}},\mathbf{x},\mathbf{\Omega}) &= \psi_{\text{max}}(\mathbf{x},\mathbf{\Omega})\;.
\end{align}
\end{subequations}
The phase space of the particle density $\psi$ consists of energy $E\in [0, E_{max}]\subset\mathbb{R}_+$, space $\mathbf{x}\in D\subset\mathbb{R}^3$ and direction of flight $\mathbf{\Omega}\in\mathbb{S}^2$. We use $S_t:\mathbb{R}_+\rightarrow\mathbb{R}_+$ to denote the stopping power, which describes the rate at which particles lose energy. The tissue density of the patient is $\rho:D\rightarrow\mathbb{R}_+$. Material cross sections $\Sigma_t$ and $\Sigma_s$ describe scattering and absorption interactions of particles with tissue. Stopping power and material cross-sections are given from physical databases \cite{gregoire2011state}.
The quantity of interest is the dose absorbed by the tissue, which can be determined from
\begin{align}
    D(\mathbf{x})=\frac{1}{\rho(\mathbf{x})}\int_0^{\infty}\int_{\mathbb{S}^2}\rho(\mathbf{x})S_t(E)\psi(E,\mathbf{x},\mathbf{\Omega})\,d\mathbf{\Omega} dE\;.
\end{align}

Since particles are assumed to lose energy continuously, the energy can be interpreted as a pseudo-time $t$. By performing a transformation of cross-sections and the particle density, cf. \cite{berthon2011numerical}, the CSD equation \eqref{eq:CSD1} becomes 
\begin{align}\label{eq:CSD2}
\partial_t \widetilde \psi = -\mathbf{\Omega}\cdot\nabla_{\mathbf{x}} \frac{\widetilde{\psi}}{\rho}-\widetilde\Sigma_t\widetilde{\psi} + \int_{\mathbb{S}^2}\widetilde\Sigma_s(t,\mathbf{\Omega}\cdot\mathbf{\Omega}')\widetilde{\psi}(t,\mathbf{x},\mathbf{\Omega}')\,d\mathbf{\Omega}' := R(t,\widetilde{\psi})\;.
\end{align} 
We omit phase space dependencies for ease of presentation. The pseudo-time $t$ is defined as
\begin{align}\label{eq:TildeE}
    t(E) := \int_0^{E_{\text{max}}} \frac{1}{S(E')}\,dE'-\int_0^{E} \frac{1}{S(E')}\,dE'\;.
\end{align}
A tilde is used to denote the transformation of particle densities and cross sections, given by $\widetilde\Sigma_t(t) =\Sigma_t(E(t))$, $\widetilde\Sigma_s(t,\mathbf{\Omega}\cdot\mathbf{\Omega}') =\Sigma_s(E(t),\mathbf{\Omega}\cdot\mathbf{\Omega}')$ and
\begin{align}
    \widetilde{\psi}(t,\mathbf{x},\mathbf{\Omega}):= S(E(t))\rho(\mathbf{x})\psi(E(t),\mathbf{x},\mathbf{\Omega})\;.
\end{align}
To simplify our presentation, we will from now on omit the tilde and always refer to the transformed quantities by $\psi$, $\Sigma_t$ and $\Sigma_s$. Furthermore, the in-scattering operator will be denoted as $\mathcal{S}$, i.e.,
\begin{align}
(\mathcal{S}\psi)(t,\mathbf{x},\mathbf{\Omega}):= \int_{\mathbb{S}^2}\Sigma_s(t,\mathbf{\Omega}\cdot\mathbf{\Omega}')\psi(t,\mathbf{x},\mathbf{\Omega}')\,d\mathbf{\Omega}'\;.
\end{align} 
\subsection{Collision source method}
First collision source methods split the solution to the radiation transport equation \eqref{eq:CSD2} into a collided and uncollided part, see e.g. \cite{alcouffe1990first,hauck2013collision}. Let us write $\psi(t,\mathbf{x},\mathbf{\Omega}) = \psi_u(t,\mathbf{x},\mathbf{\Omega})+\psi_c(t,\mathbf{x},\mathbf{\Omega})$, where $\psi_u$ represents uncollided and $\psi_c$ represents collided particles. Then, the radiation transport equation \eqref{eq:CSD2} can be split into 
\begin{subequations}\label{eq:CSDSplit}
\begin{align}
\partial_t \psi_u &= -\mathbf{\Omega}\cdot\nabla_{\mathbf{x}} \frac{\psi_u}{\rho}-\Sigma_t\psi_u := R_u(t,\psi_u)\;, \label{eq:CSDSplitUnc}\\
\partial_t \psi_c &= -\mathbf{\Omega}\cdot\nabla_{\mathbf{x}} \frac{\psi_c}{\rho}-\Sigma_t\psi_c + \mathcal{S}\left(\psi_u+\psi_c\right) := R_c(t,\psi_u,\psi_c)\;.\label{eq:CSDSplitCol}
\end{align} 
\end{subequations}
The first equation describes the dynamics of uncollided particles. Since collided particles cannot generate or deplete uncollided particles, \eqref{eq:CSDSplitUnc} solely depends on $\psi_u$. Furthermore, as uncollided particles do not generate inscattering, only the outscattering term $-\Sigma_t\psi_u$ describes interactions with the background material. The dynamics of collided particles is described by \eqref{eq:CSDSplitCol} and includes inscattering from uncollided particles. This methodology of representing the solution in terms of collided and uncollided particles can be developed further: Denoting the particles that have collided $\ell = 0,\cdots,L$ times as $\psi_{\ell}$ and particles that have collided more than $L$ times as $\psi_c$, we have $\psi = \psi_0 + \psi_1 + \cdots + \psi_M + \psi_c$. Then, we obtain the equations
\begin{subequations}\label{eq:CSDSplitMCollided}
\begin{align}
\partial_t \psi_0 &= -\mathbf{\Omega}\cdot\nabla_{\mathbf{x}} \frac{\psi_0}{\rho}-\Sigma_t\psi_0 := R_0(t,\psi_0)\;, \label{eq:CSDSplitMUnc}\\
\partial_t \psi_{\ell} &= -\mathbf{\Omega}\cdot\nabla_{\mathbf{x}} \frac{\psi_{\ell}}{\rho}-\Sigma_t\psi_{\ell} + \mathcal{S}\psi_{\ell-1} := R_1(t,\psi_{\ell-1},\psi_{\ell})\;, \qquad\text{ for }\ell = 1,\cdots,L, \label{eq:CSDSplitM}\\
\partial_t \psi_c &= -\mathbf{\Omega}\cdot\nabla_{\mathbf{x}} \frac{\psi_c}{\rho}-\Sigma_t\psi_c + \mathcal{S}\left(\psi_{L}+\psi_c\right) := R_c(t,\psi_M,\psi_c)\;.\label{eq:CSDSplitMCol}
\end{align} 
\end{subequations}
\subsection{Angular discretization}
To allow for numerical approximations of solutions to the presented equations, time, space and angle need to be discretized. Numerical artifacts that numerical solutions exhibit crucially depend on the angular discretization. Deterministic discretizations can be classified into nodal and modal methods. A conventional nodal method is the \textit{discrete ordinates} (S$_N$) \cite{lewis1984computational} method, which evolves the solution on a chosen angular quadrature set. A conventional modal discretization technique is the \textit{spherical harmonics} (P$_N$) method \cite{case1967linear} which spans the solution in terms of spherical harmonics basis functions. For degree $\ell$ and order $k$, the spherical harmonics for an angle $\mathbf{\Omega} = (\sqrt{1-\mu^2}\cos{\varphi},\sqrt{1-\mu^2}\sin{\varphi},\mu)^T$ are defined as
\begin{align*}
Y_{\ell}^k(\mathbf{\Omega}) = \sqrt{\frac{2\ell +1}{4\pi}\frac{(\ell-k)!}{(\ell+k)!}}\ e^{ik\varphi}P_{\ell}^k(\mu)\;,
\end{align*}
where $P_{\ell}^k$ are the associated Legendre polynomials. In this work, we use the real spherical harmonics
\begin{align*}
    m_{\ell}^k = 
    \begin{cases}
        \frac{(-1)^k}{\sqrt{2}}\left( Y_{\ell}^k + (-1)^k Y_{\ell}^{-k} \right), & k > 0\;, \\
        Y_{\ell}^0 & k = 0 \;, \\
        -\frac{(-1)^k i}{\sqrt{2}}\left( Y_{\ell}^{-k} - (-1)^k Y_{\ell}^{k} \right), & k < 0\;.
    \end{cases}
\end{align*}
Let us collect all basis functions up to degree $N$ in a vector
\begin{align*}
    \mathbf m = (m_0^0, m_1^{-1}, m_1^{0}, m_1^{1},\cdots, m_N^{N})^T\in\mathbb{R}^{(N+1)^2}
\end{align*}
and choose the modal approximation $\psi(t,\mathbf{x},\mathbf{\Omega}) \approx \mathbf{u}(t,\mathbf x)^T\mathbf{m}(\mathbf{\Omega})$. The coefficient (or moment) vector $\mathbf{u}$ has $m:=(N+1)^2$ entries. Then, the P$_N$ equations for $\mathbf{x} = (x,y,z)^T$ read
\begin{align*}
    \partial_t \mathbf u (t,\mathbf x) =-\mathbf A\cdot\nabla_{\mathbf{x}} \frac{\mathbf u(t,\mathbf x)}{\rho(\mathbf x)}-\Sigma_t(t)\mathbf u (t,\mathbf x)+\bm{\Sigma}\mathbf u (t,\mathbf x),
\end{align*}
where $\mathbf A\cdot\nabla_{\mathbf{x}} := \mathbf A_1\partial_{x} + \mathbf A_2\partial_y+ \mathbf A_3\partial_z$ with $\mathbf A_i := \int_{\mathbb{S}^2}\mathbf m\mathbf m^T \Omega_i \, d\bm\Omega$. The diagonal in-scattering matrix $\bm{\Sigma}$ has entries $\Sigma_{kk}(t) = 2\pi\int_{[-1,1]}P_{k}(\mu)\Sigma_s(t,\mu)\,d\mu$. Note that $\Sigma_{t}(t) = \Sigma_{11}(t) > 0$ and 
\begin{align}\label{eq:SigmaKK}
    \vert \Sigma_{kk}(t) \vert\leq 2\pi\int_{[-1,1]}\vert P_{k}(\mu)\vert\cdot\vert\Sigma_s(t,\mu)\vert\,d\mu \leq 2\pi\int_{[-1,1]}\vert\Sigma_s(t,\mu)\vert\,d\mu= \Sigma_{11}(t).
\end{align}

\subsection{Spatial discretization}
Several benchmarks in radiation transport assume a two-dimensional spatial domain, i.e., the spatial variable becomes $\mathbf x = (x,y)^T$. In this setting, we focus on the discretization of the P$_N$ equations, since these will become relevant in the proposed DLRA method. Then, we have
\begin{align}\label{eq:PNSystem2D}
\partial_t \mathbf u(t,\mathbf x) = -\mathbf{A}_{x}\partial_{x}\mathbf u(t,\mathbf x)-\mathbf{A}_y\partial_y\mathbf u(t,\mathbf x) + \mathbf G \mathbf u(t,\mathbf x)\;.
\end{align}
A finite volume discretization splits the spatial domain $D$ into $N_x\cdot N_y$ cells. In case of a structured quadrangle grid, the $x$ and $y$ domains are discretized into uniform  one-dimensional grids $x_{1}\leq x_2 \leq \cdots \leq x_{N_x+1}$ and $y_{1}\leq y_2 \leq \cdots \leq y_{N_y+1}$ with grid size $\Delta x$ and $\Delta y$ respectively. Then, the cell of index $(i,j)$ is defined on $I_{ij}:=[x_i,x_{i+1}]\times [y_j,y_{j+1}]$ on which the numerical solution is chosen as
\begin{align*}
\mathbf u_{ij}(t) \simeq \frac1{\Delta x \Delta y}\int_{I_{ij}} \mathbf u(t,\mathbf x) \,d\mathbf x\;.
\end{align*}
In the same manner we compute the patient density $\rho_{ij}$. To simplify the presentation of the spatial discretization, let us collect $\mathbf u_{ij}(t) = \left(u_{ijk}(t)\right)_{k=1}^m\in\mathbb{R}^{m}$ for $i = 1,\cdots,N_x$ and $j = 1,\cdots, N_y$ into a matrix $\textbf u(t)\in\mathbb{R}^{n_x\times m}$, where $n_x := N_x \cdot N_y$. For this, we define the function $\text{idx}:\mathbb{N}\times\mathbb{N}\rightarrow \mathbb{N}$ as $\text{idx}(i,j) = (i-1)\cdot N_x + j$. Then, the entries of $\textbf u(t)$ are defined as $u_{\text{idx}(i,j),k}(t) = u_{ijk}(t)$.

With this notation at hand, we can define a finite volume scheme for \eqref{eq:PNSystem2D} in compact notation. Let us define the sparse diffusion stencil matrices $\mathbf L_{x,y}^{(1)}\in \mathbb{R}^{n_x \times n_x}$ as
\begin{align*}
&L^{(1)}_{x,\text{idx}(i,j),\text{idx}(i,j)}= \frac{1}{\rho_{ij}\Delta x}\;, \quad L^{(1)}_{x,\text{idx}(i,j),\text{idx}(i\pm 1,j)}= -\frac{1}{2\rho_{i\pm 1,j}\Delta x}\;, \\
&L^{(1)}_{y,\text{idx}(i,j),\text{idx}(i,j)}= \frac{1}{\rho_{ij}\Delta y}\;, \quad L^{(1)}_{y,\text{idx}(i,j),\text{idx}(i,j\pm 1)}= -\frac{1}{2\rho_{i,j\pm 1}\Delta y}\;,
\end{align*}
as well as the sparse advection stencil matrices $\mathbf L_{x,y}^{(2)}\in \mathbb{R}^{n_x \times n_x}$ as
\begin{align*}
L^{(2)}_{x,\text{idx}(i,j),\text{idx}(i\pm 1,j)}= \frac{\pm 1}{2\rho_{i\pm 1,j}\Delta x}\;,\quad L^{(2)}_{y,\text{idx}(i,j),\text{idx}(i,j\pm 1)}= \frac{\pm 1}{2\rho_{i,j\pm 1}\Delta y}\;.
\end{align*}
Furthermore with $\mathbf A_{x,y} = \mathbf V_{x,y} \mathbf\Lambda_{x,y} \mathbf V_{x,y}^{T}$, the Roe matrices $\vert \mathbf A_{x} \vert$, $\vert \mathbf A_{y} \vert\in\mathbb{R}^{m\times m}$ are defined as 
\begin{align*}
\vert \mathbf A_{x} \vert := \mathbf V_{x} \vert \mathbf\Lambda_{x} \vert \mathbf V_{x}^{T}\;, \text{ and } \quad \vert \mathbf V_{y} \vert := \mathbf V_{y} \vert \mathbf\Lambda_{y} \vert \mathbf V_{y}^{T}.
\end{align*}
Here, $\mathbf V_{x,y}$ collects the orthonormal eigenvectors of the symmetric matrices $\mathbf A_{x,y}$ and $\mathbf\Lambda_{x,y} = \text{diag}(\lambda_{1}^{x,y},\cdots \lambda_{m}^{x,y})$ are the corresponding real eigenvalues. Then, the semi-discrete finite volume update becomes a huge matrix differential equation of the form
\begin{align*}
\dot{\mathbf u}(t) = \mathbf F(\mathbf u(t)) + \mathbf G(t, \mathbf u(t),\mathbf u(t)),
\end{align*}
where
\begin{align*}
\mathbf F(\mathbf u) :=& \ \mathbf L_x^{(2)} \mathbf u \mathbf A_{x}^T + \mathbf L_y^{(2)} \mathbf u \mathbf A_{y}^T + \mathbf L_x^{(1)} \mathbf u \vert\mathbf A_{x}\vert^T+ \mathbf L_y^{(1)} \mathbf u \vert\mathbf A_{y}\vert^T,\\
\mathbf G(t, \mathbf v,\mathbf u) :=& \ -\Sigma_t(t)\mathbf u + \mathbf v \bm{\Sigma}(t)\;.
\end{align*}
The costs of evaluating the right-hand side are $C_{\text{P}_N} \lesssim n_x \cdot m$, when accounting for the sparsity of all stencil, flux and Roe matrices which leads to linear costs in $n_x$ and $m$ of matrix products. 

\section{Recap: Dynamical low-rank approximation}\label{sec:backgroundsDLRA}
\subsection{Main framework}
This section gives a brief overview on dynamical low-rank approximation \cite{KochLubich07} for matrix differential equations $\dot{\mathbf u}(t) = \mathbf F(t, \mathbf u(t))$. Dynamical low-rank approximation represent and evolves the solution on a manifold of rank $r$ matrices, which we denote by $\mathcal{M}_r$. A low-rank representation is given by the SVD-like factorization
\begin{align}\label{eq:rankrsol}
\mathbf u(t)\approx\mathbf{X}(t)\mathbf{S}(t)\mathbf{W}(t)^T,
\end{align}
where $\mathbf{X}\in\mathbb{R}^{n_x\times r}$ and $\mathbf{W}\in\mathbb{R}^{m\times r}$ are basis matrices with orthonormal column vectors and $\mathbf{S}\in\mathbb{R}^{r\times r}$ is a dense coefficient matrix. Time evolution equations for the factors can be defined by imposing
\begin{align}\label{eq:DLRproblem}
\dot{\mathbf u}(t)\in T_{ \mathbf u(t)}\mathcal{M}_r \qquad \text{such that} \qquad \left\Vert \dot{\mathbf u}(t)-\mathbf{F}(t,\mathbf u(t)) \right\Vert = \text{min}\;.
\end{align}
We use $T_{\mathbf u(t)}\mathcal{M}_r$ to denote the tangent space of $\mathcal{M}_r$ at $\mathbf u(t)$, i.e., the solution should remain of rank $r$ over time while minimizing the defect. These conditions yield a time evolution equation for the low-rank solution \cite[Lemma~4.1]{KochLubich07}, which reads
\begin{align}\label{eq:lowRankProjector}
\dot{\mathbf u}(t) = \mathbf{P}(\mathbf u(t))\mathbf{F}(t,\mathbf u(t)).
\end{align} 
The operator $\mathbf{P}$ is the orthogonal projection onto the tangent space, given by
\begin{align*}
\mathbf P\mathbf g = \mathbf{X}\mathbf{X}^T \mathbf g - \mathbf{X}\mathbf{X}^T \mathbf g \mathbf{W}\mathbf{W}^T + \mathbf g \mathbf{W}\mathbf{W}^T.
\end{align*}
Evolution equations can then be derived for the factors $\mathbf{X}$, $\mathbf{S}$ and $\mathbf{W}$, see \cite{KochLubich07}. However, the resulting equations depend on the inverse of the commonly ill conditioned coefficient matrix, which substantially limits the permitted step size \cite{KieriLubichWalach}. 

\subsection{Robust fixed rank integrators}
Two robust integrators have been proposed for the evolution equation \eqref{eq:lowRankProjector}. First, the matrix projector--splitting integrator \cite{LubichOseledets} splits \eqref{eq:lowRankProjector} by a Lie-Trotter splitting technique, yielding
\begin{subequations}\label{eq:projectorSplitEq}
\begin{alignat}{2}
\dot{\mathbf u}_{\RomanNumeralCaps{1}}(t) &= \mathbf{F}(\mathbf u_{\RomanNumeralCaps{1}}(t))\mathbf{W}\mathbf{W}^T, \quad && \mathbf u_{\RomanNumeralCaps{1}}(t_0) = \mathbf u(t_0), \label{eq:DLR1}\\ 
\dot{\mathbf u}_{\RomanNumeralCaps{2}}(t) &= -\mathbf{X}\mathbf{X}^T\mathbf{F}(\mathbf u_{\RomanNumeralCaps{2}}(t))\mathbf{W}\mathbf{W}^T, \quad && \mathbf u_{\RomanNumeralCaps{2}}(t_0) = \mathbf u_{\RomanNumeralCaps{1}}(t_1),\label{eq:DLR2}\\ 
\dot{\mathbf u}_{\RomanNumeralCaps{3}}(t) &= \mathbf{X}\mathbf{X}^T \mathbf{F}(\mathbf u_{\RomanNumeralCaps{3}}(t)) , \quad && \mathbf u_{\RomanNumeralCaps{3}}(t_0) = \mathbf u_{\RomanNumeralCaps{2}}(t_1).\label{eq:DLR3}
\end{alignat}
\end{subequations}
The resulting consecutive movement in the low-rank manifold ensures robustness irrespective of singular
values and thereby allows for increased step sizes \cite{LubichOseledets}. Defining the decompositions $\mathbf u_{\RomanNumeralCaps{1}} = \mathbf{K}\mathbf{W}^T$ as well as $\mathbf u_{\RomanNumeralCaps{3}} = \mathbf{X}\mathbf{L}$ gives the \textit{matrix projector-splitting integrator}, which updates the low-rank factors $\mathbf X^{0} = \mathbf X(t_0)$, $\mathbf W^{0} = \mathbf W(t_0)$ and $\mathbf S^{0} = \mathbf S(t_0)$ to time $t_1 = t_0 + \Delta t$:
\begin{enumerate}
    \item \textbf{$K$-step}: Update $\mathbf X^{0}$ to $\mathbf X^{1}$ and $\mathbf S^0$ to $\mathbf{\widetilde S}^0$ via
\begin{align}
\dot{\mathbf K}(t) &= \mathbf{F}(\mathbf{K}(t)\mathbf{W}^{0,T})\mathbf{W}^0, \qquad \mathbf K(t_0) = \mathbf{X}^0\mathbf{S}^0.\label{eq:KStepSemiDiscrete}
\end{align}
Determine $\mathbf X^1$ and $\mathbf{\widetilde S}^0$ with $\mathbf K(t_1) = \mathbf X^1 \mathbf{\widetilde S}^0$ by performing a QR decomposition.
\item \textbf{$S$-step}: Update $\mathbf{\widetilde S^0}$ to $\mathbf{\widetilde S^1}$ via
\begin{align}
\dot{\mathbf{\widetilde S}}(t) = -\mathbf{X}^{1,T}\mathbf{F}(\mathbf{X}^{1}\mathbf{\widetilde S}(t)\mathbf{W}^{0,T})\mathbf{W}^0, \qquad \mathbf{\widetilde S}(t_0) = \mathbf{\widetilde S}^0\label{eq:SStepSemiDiscrete}
\end{align}
and set $\mathbf{\widetilde S}^1 = \mathbf{\widetilde S}(t_1)$.
\item \textbf{$L$-step}: Update $\mathbf W^0$ to $\mathbf W^1$ and $\mathbf{\widetilde S}^1$ to $\mathbf S^1$ via
\begin{align}
\dot{\mathbf L}(t) &= \mathbf{X}^{1,T}\mathbf{F}(\mathbf{X}^1\mathbf{L}(t)), \qquad \mathbf L(t_0) = \mathbf{\widetilde S}^1 \mathbf{W}^{0,T}.\label{eq:LStepSemiDiscrete}
\end{align}
Determine $\mathbf W^1$ and $\mathbf S^1$ with $\mathbf L(t_1) = \mathbf S^1 \mathbf W^{1,T}$ by performing a QR decomposition.
\end{enumerate}
Then, the time updated solution is $\mathbf{u}(t_1) = \mathbf{X}^1\mathbf{S}^1\mathbf{W}^{1,T}$. It has been noted in \cite{kusch2021stability} that when the flux function takes the form $\mathbf F (t,\mathbf u(t)) = \mathbf u(t)\mathbf G(t)$, the $K$ and $S$-steps cancel out and only the $L$-step determines the dynamics.

The second robust integrator is the \textit{unconventional integrator}, which has recently been introduced in~\cite{CeL21}. This integrator first performs basis updates of $\textbf X$ and $\textbf W$ in parallel and then updates the coefficient matrix $\textbf S$  by a Galerkin step. This integrator shares the robustness properties of the matrix projector--splitting integrator \cite{CeL21}. It takes the form
\begin{enumerate}
    \item \textbf{$K$-step}: Update $\mathbf X^{0}$ to $\mathbf X^{1}$ via
    \begin{align}
        \dot{\mathbf K}(t) &= \mathbf{F}(\mathbf{K}(t)\mathbf{W}^{0,T})\mathbf{W}^0, \qquad \mathbf K(t_0) = \mathbf{X}^0\mathbf{S}^0.\label{eq:KStepSemiDiscreteUI}
    \end{align}
Determine $\mathbf X^1$ with $\mathbf K(t_1) = \mathbf X^1 \mathbf R$ and store $\mathbf M = \mathbf X^{1,T}\mathbf X^0$.
\item \textbf{$L$-step}: Update $\mathbf W^0$ to $\mathbf W^1$ via
\begin{align}
\dot{\mathbf L}(t) &= \mathbf{X}^{0,T}\mathbf{F}(\mathbf{X}^0\mathbf{L}(t)), \qquad \mathbf L(t_0) = \mathbf{S}^0 \mathbf{W}^{0,T}.\label{eq:LStepSemiDiscreteUI}
\end{align}
Determine $\mathbf W^1$ with $\mathbf L(t_1) = \mathbf W^1\mathbf{\widetilde R}$ and store $\mathbf N = \mathbf W^{1,T} \mathbf W^0$.
\item \textbf{$S$-step}: Update $\mathbf S^0$ to $\mathbf S^1$ via
\begin{align}
\dot{\mathbf S}(t) = \mathbf{X}^{1,T}\mathbf{F}(\mathbf{X}^{1}\mathbf{S}(t)\mathbf{W}^{1,T})\mathbf{W}^1, \qquad \mathbf S(t_0) &= \mathbf M\mathbf S^0 \mathbf N^T\label{eq:SStepSemiDiscreteUI}
\end{align}
and set $\mathbf S^1 = \mathbf S(t_1)$.
\end{enumerate}

\subsection{Rank adaptive unconventional integrator}
The unconventional integrator has recently been extended to allow for rank adaptivity \cite{ceruti2021rank}. I.e., given a tolerance parameter $\vartheta$, the integrator adapts the rank in time. 

Starting from time $t_0$ where the solution has rank $r_0$, the integrator gives the factored solution at time $t_1$ with rank $r_1 \leq 2r_0$. In the following, we use $r = r_0$ and use hats to denote matrices of rank $2r$. Then the rank adaptive integrator reads
\begin{enumerate}
    \item \textbf{$K$-step}: Update $\mathbf X^{0}\in\mathbb{R}^{n_x\times r}$ to $\mathbf{\widehat X}^{1}\in\mathbb{R}^{n_x\times 2r}$ via
    \begin{align}
        \dot{\mathbf K}(t) &= \mathbf{F}(\mathbf{K}(t)\mathbf{W}^{0,T})\mathbf{W}^0, \qquad \mathbf K(t_0) = \mathbf{X}^0\mathbf{S}^0.\label{eq:KStepSemiDiscreteUIRankadapt}
    \end{align}
Determine $\mathbf{\widehat X}^1$ with $[\mathbf K(t_1), \mathbf{X}^0]  = \mathbf{\widehat X}^1 \mathbf R$ and store $\mathbf{\widehat M} = \mathbf{\widehat X}^{1,T}\mathbf X^0\in\mathbb{R}^{2r\times r}$.
\item \textbf{$L$-step}: Update $\mathbf W^0\in\mathbb{R}^{m\times r}$ to $\mathbf{\widehat W}^1\in\mathbb{R}^{m\times 2r}$ via
\begin{align}
\dot{\mathbf L}(t) &= \mathbf{X}^{0,T}\mathbf{F}(\mathbf{X}^0\mathbf{L}(t)), \qquad \mathbf L(t_0) = \mathbf{S}^0 \mathbf{W}^{0,T}.\label{eq:LStepSemiDiscreteUIRankadapt}
\end{align}
Determine $\mathbf{\widehat W}^1$ with $[\mathbf L(t_1),\mathbf W^0] = \mathbf{\widehat W}^1\mathbf{\widetilde R}$ and store $\mathbf{\widehat N} = \mathbf{\widehat W}^{1,T} \mathbf W^0$.
\item \textbf{$S$-step}: Update $\mathbf S^0\in\mathbb{R}^{r\times r}$ to $\mathbf{\widehat S}^1\in\mathbb{R}^{2r\times 2r}$ via
\begin{align}
\dot{\mathbf{\widehat S}}(t) = \mathbf{\widehat X}^{1,T}\mathbf{F}(\mathbf{\widehat X}^1\mathbf{\widehat S}(t)\mathbf{\widehat W}^{1,T})\mathbf{\widehat W}^{1}, \qquad \mathbf{\widehat S}(t_0) &= \mathbf{\widehat M}\mathbf S^0 \mathbf{\widehat N}^T\label{eq:SStepSemiDiscreteUIRankadapt}
\end{align}
and set $\mathbf{\widehat S}^1 = \mathbf{\widehat S}(t_1)$.
\item \textbf{Truncation}:
		Determine the SVD $\; \bm{\widehat{S}}^{1}= \bm{\widehat{P}}  \bm{\widehat{\Sigma}} \bm{\widehat{Q}}^\top$ where $\bm{\widehat{\Sigma}}=\text{diag}(\sigma_j)$. For a given tolerance~$\vartheta$, choose the new rank $r_{1}\le 2r$ such that
		$$
		\biggl(\ \sum_{j=r_{1}+1}^{2r} \sigma_j^2 \biggr)^{1/2} \le \vartheta.
		$$
		Compute the new factors as follows:
Let $\bm{S}^1$ be the $r_1\times r_1$ diagonal matrix with the $r_1$ largest singular values and let $\bm{P}_1\in \mathbb{R}^{2r\times r_1}$ and $\bm{Q}_1\in \mathbb{R}^{2r\times r_1}$ contain the first $r_1$ columns of $\bm{\widehat{P}}$ and $\bm{\widehat{Q}}$, respectively. Finally, set $\bm{X}^1 = \bm{\widehat{X}}^1 \bm{P}_1\in \mathbb{R}^{m\times r_1}$ and
$\bm{W}^1 = \bm{\widehat{W}}^1 \bm{Q}_1 \in \mathbb{R}^{n\times r_1}$.
\end{enumerate}

\section{A robust collision source method for dynamical low-rank approximation}\label{sec:mainMethod}
In this section we present the main method, which aims at providing an efficient and robust alternative to conventional strategies. Key ingredients for the construction are 1) a collision source method to define a splitting of the original equation, 2) a further splitting of collision terms which are treated implicitly, 3) computing individual DLRA updates by using the unconventional integrator for collided particles.
\subsection{Collided-uncollided split}
We start by deriving moment equations for the collided particles in the collision source method and use an S$_N$ method for uncollided particles. Without going into detail, we denote the S$_N$ solution of the uncollided particles as $\bm \psi$ and the right-hand side of a time continuous S$_N$ method for streaming (i.e., particles move without interacting with tissue) as $\mathbf F_S(\bm{\psi})$. Then the evolution equations for the semi-discrete solution become
\begin{subequations}\label{eq:CSDSplitMCollidedPN}
\begin{align}
\bm{\dot{\psi}}(t) &= \mathbf F_S(\bm{\psi}(t))-\Sigma_t(t)\bm{\psi}(t)\;, \label{eq:CSDSplitMUncPN}\\
\mathbf{\dot{u}}_{1}(t) &= \mathbf F(\mathbf u_{1}(t))-\Sigma_t(t)\mathbf u_{1}(t) + \bm{\psi}(t)\mathbf{T}_{M}\bm{\Sigma}(t)\;, \label{eq:CSDSplitM1}\\
\mathbf{\dot{u}}_{\ell}(t) &= \mathbf F(\mathbf u_{\ell}(t)) -\Sigma_t(t)\mathbf u_{\ell}(t) + \mathbf u_{\ell-1}(t)\bm{\Sigma}(t)\;,\qquad\text{ for }\ell = 2,\cdots,L, \label{eq:CSDSplitM}\\
\mathbf{\dot{u}}_{c}(t) &= \mathbf F(\mathbf u_{c}(t))-\Sigma_t(t)\mathbf u_{c} + \left(\mathbf u_{L}(t)+\mathbf u_c(t)\right)\bm{\Sigma}(t)\;,\label{eq:CSDSplitMCol}
\end{align} 
\end{subequations}
where the matrix $\mathbf{T}_{M}$ maps the nodal solution onto its moments. These $L+2$ equations can be solved consecutively. Since radiation therapy commonly investigates the effects of particle beams which enter the computational domain from the boundary, the S$_N$ equations from the uncollided particles can be solved efficiently by using a biased quadrature rule. I.e., the quadrature only encodes the small number of possible flight directions. After the first collision, particles move into all directions. To account for the increased complexity, we describe the collided solution through dynamical low-rank approximation. Following \cite{kusch2021stability}, we split streaming and scattering parts and use the matrix projector--splitting integrator to update in-scattering in equation \eqref{eq:CSDSplitMCol} as well as the unconventional integrator for the remainder. 

To simplify our presentation, we start by discussing this strategy for \eqref{eq:CSDSplitM1} and then extend it to the remaining equations. Let us first split \textit{streaming} and \textit{scattering} in \eqref{eq:CSDSplitM1}. Omitting the subscript $1$ gives
\begin{subequations}\label{eq:splitStreamScatter}
\begin{alignat}{2}
\mathbf{\dot{u}}_{\RomanNumeralCaps{1}}(t) &= \mathbf F(\mathbf u_{\RomanNumeralCaps{1}}(t))\;,\qquad && \mathbf{u}_{\RomanNumeralCaps{1}}(t_0) = \bm{u}_1(t_0)\;,\label{eq:PNStreaming} \\
\mathbf{\dot{u}}_{\RomanNumeralCaps{2}}(t) &= -\Sigma_t(t)\mathbf u_{\RomanNumeralCaps{2}}(t) + \bm{\psi}(t_1)\mathbf{T}_{M}\bm{\Sigma}(t)\;,\qquad && \mathbf{u}_{\RomanNumeralCaps{2}}(t_0) = \mathbf{u}_{\RomanNumeralCaps{1}}(t_1)\;.\label{eq:PNScattering}
\end{alignat}
\end{subequations}
The updated solution at time $t_1 = t_0 + \Delta t$ is then given as $\bm{u}_1(t_1) = \mathbf{u}_{\RomanNumeralCaps{2}}(t_1)$. Note that the splitting method introduces an error of $O(\Delta t)$, which can be reduced by high order splitting methods. We start with the derivation of the basis update and Galerkin step equations of the unconventional integrator for the streaming part \eqref{eq:PNStreaming}. I.e., we derive evolution equations for $\mathbf X_{\RomanNumeralCaps{1}}(t),\mathbf S_{\RomanNumeralCaps{1}}(t),\mathbf W_{\RomanNumeralCaps{1}}(t)$ such that $\mathbf{u}_{\RomanNumeralCaps{1}}(t)\approx\mathbf X_{\RomanNumeralCaps{1}}(t)\mathbf S_{\RomanNumeralCaps{1}}(t)\mathbf W_{\RomanNumeralCaps{1}}(t)^T$. To simplify notation let us omit Roman indices in the following.
The $K$-step equations \eqref{eq:KStepSemiDiscreteUI} read
\begin{align}\label{eq:KStreaming}
\mathbf{\dot{K}}(t) =& \ \mathbf F(\mathbf{K}(t)\mathbf{W}^{0,T})\mathbf{W}^0 \nonumber\\
=& \ \mathbf L_x^{(2)} \mathbf{K}(t)\mathbf{W}^{0,T} \mathbf A_{x}^T\mathbf{W}^0 + \mathbf L_y^{(2)} \mathbf{K}(t)\mathbf{W}^{0,T} \mathbf A_{y}^T\mathbf{W}^0\nonumber\\
& + \mathbf L_x^{(1)} \mathbf{K}(t)\mathbf{W}^{0,T} \vert\mathbf A_{x}\vert^T\mathbf{W}^0+ \mathbf L_y^{(1)} \mathbf{K}(t)\mathbf{W}^{0,T} \vert\mathbf A_{y}\vert^T\mathbf{W}^0 \nonumber\\
=& \ \mathbf L_x^{(2)} \mathbf{K}(t)\mathbf{\widehat A}_{x}^0 + \mathbf L_y^{(2)} \mathbf{K}(t)\mathbf{\widehat A}_{y}^0 + \mathbf L_x^{(1)} \mathbf{K}(t)\vert\mathbf{\widehat A}_{x}\vert^0+ \mathbf L_y^{(1)} \mathbf{K}(t)\vert\mathbf{\widehat A}_{y}\vert^0,
\end{align}
where we use $\mathbf{\widehat A}_{x,y}^0 := \mathbf{W}^{0,T} \mathbf A_{x,y}^T\mathbf{W}^0$ and $\vert\mathbf{\widehat A}_{x,y}\vert^0:=\mathbf{W}^{0,T} \vert\mathbf A_{x,y}\vert^T\mathbf{W}^0$. The numerical costs to compute these matrices are of $O(r^2\cdot m)$ and evaluating the right-hand side of the $K$-step equations has costs of $O(r^2\cdot n_x)$. To point out that the spatial basis is not yet updated by the scattering step \eqref{eq:PNScattering}, we define the superscript $\inter$, i.e., the solution is given as $\mathbf{X}^{\inter} := \mathbf{X}_{\RomanNumeralCaps{1}}(t_1)$.

The $L$-step equations \eqref{eq:LStepSemiDiscreteUI} read
\begin{align}\label{eq:LStreaming}
\mathbf{\dot{L}}(t) =& \ \mathbf{X}^{0,T}\mathbf F(\mathbf{X}^0\mathbf{L}(t))\nonumber\\
=& \ \mathbf{X}^{0,T}\mathbf L_x^{(2)} \mathbf{X}^0\mathbf{L}(t) \mathbf A_{x}^T + \mathbf{X}^{0,T}\mathbf L_y^{(2)} \mathbf{X}^0\mathbf{L}(t) \mathbf A_{y}^T \nonumber\\
&+ \mathbf{X}^{0,T}\mathbf L_x^{(1)} \mathbf{X}^0\mathbf{L}(t) \vert\mathbf A_{x}\vert^T+ \mathbf{X}^{0,T}\mathbf L_y^{(1)} \mathbf{X}^0\mathbf{L}(t) \vert\mathbf A_{y}\vert^T\nonumber\\
=& \ \mathbf{\widehat L}_x^{(2),0} \mathbf{L}(t) \mathbf A_{x}^T + \mathbf{\widehat L}_y^{(2),0}\mathbf{L}(t) \mathbf A_{y}^T + \mathbf{\widehat L}_x^{(1),0} \mathbf{L}(t) \vert\mathbf A_{x}\vert^T+ \mathbf{\widehat L}_y^{(1),0}\mathbf{L}(t) \vert\mathbf A_{y}\vert^T\;,
\end{align}
where we use $\mathbf{\widehat L}_{x,y}^{(2),0} := \mathbf{X}^{0,T}\mathbf L_{x,y}^{(2)} \mathbf{X}^0$ and $\mathbf{\widehat L}_{x,y}^{(1),0}:=\mathbf{X}^{0,T}\mathbf L_{x,y}^{(1)} \mathbf{X}^0$. The numerical costs to compute these matrices are of $O(r^2\cdot n_x)$ and evaluating the right-hand side of the $L$-step equations has costs of $O(r^2\cdot~m)$. To point out that the directional basis is not yet updated by the scattering step \eqref{eq:PNScattering}, we define $\mathbf{W}^{\inter} := \mathbf{W}_{\RomanNumeralCaps{1}}(t_1)$.

The $S$-step equations \eqref{eq:SStepSemiDiscreteUI} read
\begin{align}\label{eq:SStreaming}
\mathbf{\dot{S}}(t) =& \ \mathbf{X}^{\inter,T}\mathbf F(\mathbf{X}^{\inter}\mathbf{S}(t)\mathbf{W}^{\inter,T})\mathbf{W}^{\inter}\nonumber \\
=& \ \mathbf{X}^{\inter,T}\mathbf L_x^{(2)} \mathbf{X}^{\inter}\mathbf{S}(t)\mathbf{W}^{\inter,T} \mathbf A_{x}^T\mathbf{W}^{\inter} + \mathbf{X}^{\inter,T}\mathbf L_y^{(2)} \mathbf{X}^{\inter}\mathbf{S}(t)\mathbf{W}^{\inter,T} \mathbf A_{y}^T\mathbf{W}^{\inter} \nonumber \\
&+ \mathbf{X}^{\inter,T}\mathbf L_x^{(1)} \mathbf{X}^{\inter}\mathbf{S}(t)\mathbf{W}^{\inter,T} \vert\mathbf A_{x}\vert^T\mathbf{W}^{\inter}+ \mathbf{X}^{\inter,T}\mathbf L_y^{(1)} \mathbf{X}^{\inter}\mathbf{S}(t)\mathbf{W}^{\inter,T} \vert\mathbf A_{y}\vert^T\mathbf{W}^{\inter}\nonumber \\
=& \ \mathbf{\widehat L}_x^{(2),\inter} \mathbf{S}(t) \mathbf{\widehat A}_{x}^{\inter} + \mathbf{\widehat L}_y^{(2),\inter}\mathbf{S}(t) \mathbf{\widehat A}_{y}^{\inter} + \mathbf{\widehat L}_x^{(1),\inter} \mathbf{S}(t) \vert\mathbf{\widehat A}_{x}\vert^{\inter}+ \mathbf{\widehat L}_y^{(1),\inter}\mathbf{S}(t) \vert\mathbf{\widehat A}_{y}\vert^{\inter}\;.
\end{align}
The numerical costs to compute update matrices and to evaluate the right-hand side of the $S$-step equations are of $O(r^2\cdot (n_x+m))$. To point out, the the coefficient matrix is not yet updated by the scattering step \eqref{eq:PNScattering}, we define $\mathbf{S}^{\inter} := \mathbf{S}_{\RomanNumeralCaps{1}}(t_1)$.

In a second step, we use the updated factors as initial condition, i.e., $\mathbf X_{\RomanNumeralCaps{2}}(t_0) = \mathbf X^{\inter}$, $\mathbf S_{\RomanNumeralCaps{2}}(t_0) = \mathbf S^{\inter}$ and $\mathbf W_{\RomanNumeralCaps{2}}(t_0) = \mathbf W^{\inter}$. Again Roman numbers are omitting in the following. However, we do include a subscript $1$ to denote that we are solving for the factors of $\mathbf{u}_1$. For the scattering step \eqref{eq:PNScattering}, determining the $K$, $L$ and $S$-steps is straightforward and leads to
\begin{subequations}\label{eq:u1ScatterKLS}
\begin{align}
    \mathbf{\dot{K}}_1(t) =& \ -\Sigma_t(t)\mathbf{K}_1(t) + \bm{\psi}(t_1)\mathbf{T}_{M}\bm{\Sigma}(t)\mathbf{W}_1^{\inter}, \\
    \mathbf{\dot{L}}_1(t) =& \ -\Sigma_t(t)\mathbf{L}_1(t) + \mathbf{X}_{1}^{{\inter},T}\bm{\psi}(t_1)\mathbf{T}_{M}\bm{\Sigma}(t), \\
    \mathbf{\dot{S}}_1(t) =& \ -\Sigma_t(t)\mathbf{S}_1(t) + \mathbf{X}_1^{1,T}\bm{\psi}(t_1)\mathbf{T}_{M}\bm{\Sigma}(t)\mathbf{W}_1^{1}.
\end{align}
\end{subequations}
The time updated solution after streaming and scattering is given as $\mathbf u_1(t_1) = \mathbf X_1(t_1)\mathbf S_1(t_1)\mathbf W_1(t_1)^T$. When the directed S$_N$ quadrature set has $n_q$ nodes, computational costs are of $O(r\cdot n_q\cdot(n_x+m))$. Note that since radiation therapy uses highly peaked particle beams as boundary conditions or source terms, only a limited number of directions needs to be resolved by the quadrature, i.e., $n_q$ is expected to be small. 

In the same manner, evolution equations for the factors of the solutions to the remaining moment equations \eqref{eq:CSDSplitM} and \eqref{eq:CSDSplitMCol} are derived. Here, the streaming update can be determined with \eqref{eq:KStreaming}, \eqref{eq:LStreaming} and \eqref{eq:SStreaming} (except for \eqref{eq:CSDSplitMUncPN}, since for $\bm{\psi}$ we use a directed S$_N$ method instead of a dynamical low-rank approximation). The scattering update for a general $\ell = 2,\cdots,L$ which we denote by a subscript reads
\begin{subequations}\label{eq:uellScatterKLS}
\begin{align*}
    \mathbf{\dot{K}}_{\ell}(t) =& \ -\Sigma_t(t)\mathbf{K}_{\ell}(t) + \mathbf{X}_{\ell-1}^{1}\mathbf{S}_{\ell-1}^1\mathbf{W}_{\ell-1}^{1,T}\bm{\Sigma}(t)\mathbf{W}_{\ell}^{\inter}, \\
    \mathbf{\dot{L}}_{\ell}(t) =& \ -\Sigma_t(t)\mathbf{L}_{\ell}(t) + \mathbf{X}_{\ell}^{{\inter},T}\mathbf{X}_{\ell-1}^{1}\mathbf{S}_{\ell-1}^1\mathbf{W}_{\ell-1}^{1,T}\bm{\Sigma}(t), \\
    \mathbf{\dot{S}}_{\ell}(t) =& \ -\Sigma_t(t)\mathbf{S}_{\ell}(t) + \mathbf{X}_{\ell}^{1,T}\mathbf{X}_{\ell-1}^{1}\mathbf{S}_{\ell-1}^1\mathbf{W}_{\ell-1}^{1,T}\bm{\Sigma}(t)\mathbf{W}_{\ell}^1.
\end{align*}
\end{subequations}
Computational costs are of $O(r^2\cdot(n_x+m))$. Lastly, for the collided solution we perform a further splitting step. Omitting the subscript $c$, we have 
\begin{subequations}\label{eq:splitStreamScatterCollided}
\begin{alignat}{2}
\mathbf{\dot{u}}_{\RomanNumeralCaps{1}}(t) &= \mathbf F(t,\mathbf u_{\RomanNumeralCaps{1}}(t))\;,\qquad && \mathbf{u}_{\RomanNumeralCaps{1}}(t_0) = \bm{u}_1(t_0)\;,\label{eq:PNStreamingCollided} \\
\mathbf{\dot{u}}_{\RomanNumeralCaps{2}}(t) &=  \mathbf{u}_{L}(t_1)\bm{\Sigma}(t)\;,\qquad && \mathbf{u}_{\RomanNumeralCaps{2}}(t_0) = \mathbf{u}_{\RomanNumeralCaps{1}}(t_1)\;,\label{eq:InCollided}\\
\mathbf{\dot{u}}_{\RomanNumeralCaps{3}}(t) &= -\Sigma_t(t)\mathbf u_{\RomanNumeralCaps{3}}(t)+\mathbf u_{\RomanNumeralCaps{3}}(t)\bm{\Sigma} \;,\qquad && \mathbf{u}_{\RomanNumeralCaps{3}}(t_0) = \mathbf{u}_{\RomanNumeralCaps{2}}(t_1)\;.\label{eq:InOutCollided}
\end{alignat}
\end{subequations}
In this case, the $K$, $S$, $L$-equations for inscattering from $\mathbf{u}_L$, i.e., equation \eqref{eq:InCollided} read (omitting Roman indices)
\begin{subequations}\label{eq:ucScatterKLS}
\begin{align}
    \mathbf{\dot{K}}_{c}(t) =& \ \mathbf{X}_{L}^{1}\mathbf{S}_{L}^1\mathbf{W}_{L}^{1,T}\bm{\Sigma}(t)\mathbf{W}_{c}^{\inter}, \\
    \mathbf{\dot{L}}_{c}(t) =& \ \mathbf{X}_{c}^{{\inter},T}\mathbf{X}_{L}^{1}\mathbf{S}_{L}^1\mathbf{W}_{L}^{1,T}\bm{\Sigma}(t), \\
    \mathbf{\dot{S}}_{c}(t) =& \ \mathbf{X}_{c}^{1,T}\mathbf{X}_{L}^{1}\mathbf{S}_{L}^1\mathbf{W}_{L}^{1,T}\bm{\Sigma}(t)\mathbf{W}_{c}^1.
\end{align}
\end{subequations}
For the in-scattering and out-scattering of the collided flux, i.e., equation \eqref{eq:InOutCollided} we use the matrix projector--splitting integrator. Following \cite{kusch2021stability}, only the $L$-step needs to be computed and we are left with
\begin{align}\label{eq:ucScatterL}
    \mathbf{\dot{L}}_{c}(t) = -\Sigma_t(t)\mathbf L_{c}(t) + \mathbf L_{c}(t)\bm{\Sigma}(t).
\end{align}
The costs for the collided particles are again $O(r^2\cdot(n_x+m))$. According to the derived steps, the scheme then consecutively updates the uncollided particles $\bm{\psi}$, the factors of $\mathbf u_1, \cdots, \mathbf u_L$ and lastly the factors of $\mathbf u_c$. In every step, the factors are first updated by a streaming step, followed by a scattering step. Lastly, when updating the factors of the collided flux, the additional $L$-step is performed to account for self-scattering. This procedure is repeated until a final time (or minimal energy) is reached. A flow chart to visualize the presented method is given in Figure~\ref{fig:flowchart}.
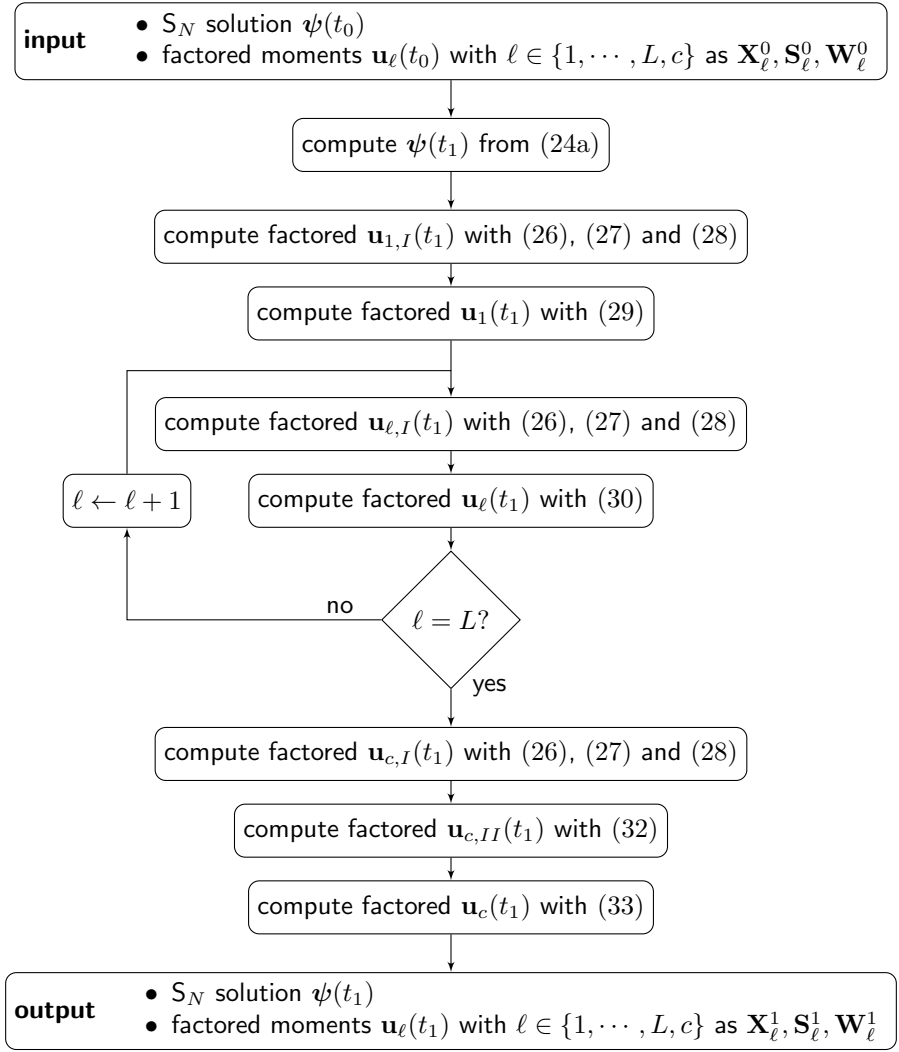
\begin{figure}[htp!]
    \centering
    \begin{tikzpicture}[node distance = 3.5cm,auto,font=\sffamily]
        \node[block](init){
        \textbf{input}
    \begin{varwidth}{\linewidth}\begin{itemize}
        \item S$_N$ solution $\bm{\psi}(t_0)$
        \item factored moments $\mathbf{u}_{\ell}(t_0)$ with $\ell \in\{1,\cdots,L,c\}$ as $\mathbf{X}^0_{\ell},\mathbf{S}^0_{\ell},\mathbf{W}^0_{\ell}$
    \end{itemize}\end{varwidth}
        };
        \node[block, below = 0.5cm of init](SN){compute $\bm{\psi}(t_1)$ from \eqref{eq:CSDSplitMUncPN}};
        \node[block, below = 0.5cm of SN](u1Stream){compute factored $\mathbf{u}_{1,\RomanNumeralCaps{1}}(t_1)$ with \eqref{eq:KStreaming}, \eqref{eq:LStreaming} and \eqref{eq:SStreaming}};
        \node[block, below = 0.3cm of u1Stream](u1Scatter){compute factored $\mathbf{u}_{1}(t_1)$ with \eqref{eq:u1ScatterKLS}};
        \node[block, below = 0.75cm of u1Scatter](uellStream){compute factored $\mathbf{u}_{\ell,\RomanNumeralCaps{1}}(t_1)$ with \eqref{eq:KStreaming}, \eqref{eq:LStreaming} and \eqref{eq:SStreaming}};
        \node[block, below = 0.3cm of uellStream](uellScatter){compute factored $\mathbf{u}_{\ell}(t_1)$ with \eqref{eq:uellScatterKLS}};
        \node[decision, below = 0.3cm of uellScatter](timeCheck){$\ell = L?$};
        \node[block, below = 0.5cm of timeCheck](ucStream){compute factored $\mathbf{u}_{c,\RomanNumeralCaps{1}}(t_1)$ with \eqref{eq:KStreaming}, \eqref{eq:LStreaming} and \eqref{eq:SStreaming}};
        \node[block, below = 0.3cm of ucStream](ucScatter){compute factored $\mathbf{u}_{c,\RomanNumeralCaps{2}}(t_1)$ with \eqref{eq:ucScatterKLS}};
        \node[block, below = 0.3cm of ucScatter](ucScatterL){compute factored $\mathbf{u}_{c}(t_1)$ with \eqref{eq:ucScatterL}};
        \node[block, left = 0.75cm of uellScatter](nUp){$\ell\leftarrow \ell+1$};
        \node[block, below = 0.5cm of ucScatterL](out){
        \textbf{output}
    \begin{varwidth}{\linewidth}\begin{itemize}
        \item S$_N$ solution $\bm{\psi}(t_1)$
        \item factored moments $\mathbf{u}_{\ell}(t_1)$ with $\ell \in\{1,\cdots,L,c\}$ as $\mathbf{X}^1_{\ell},\mathbf{S}^1_{\ell},\mathbf{W}^1_{\ell}$
    \end{itemize}\end{varwidth}
        };
        \path[line] (init) -- node [near end] {} (SN);
        \path[line] (SN) -- node [near end] {} (u1Stream);
        \path[line] (u1Stream) -- node [near end] {} (u1Scatter);
        \path[linenoarrow] (u1Scatter) -- ([yshift=0.35cm] uellStream.north);
        \path[line] (uellStream) -- node [near end] {} (uellScatter);
        \path[line] (uellScatter) -- node [near start] {} (timeCheck);
        \path[line] (timeCheck) -- node [near start] {} (ucStream);
        \path[line] (timeCheck.west) -| node [near start] {} (nUp.south);
        \path [line] (nUp.north) |- ([yshift=0.35cm] uellStream.north)-- (uellStream.north);
        \path[line] (ucStream) -- node [near start] {} (ucScatter);
        \path[line] (ucScatter) -- node [near start] {} (ucScatterL);
        \path[line] (ucScatterL) -- node [near start] {} (out);
        \node[above left=-0.5cm and 0.7cm of timeCheck](recomp){no};
        \node[above right=-1.6cm and -0.3cm of timeCheck](recomp){yes};
    \end{tikzpicture}
    \caption{Flowchart of the presented method.}
    \label{fig:flowchart}
\end{figure}
\subsection{Time (or energy) discretization}
The presented equations still continuously depend on the pseudo-time (or energy) $t$. To treat stiff scattering terms, we use an implicit time update method for the scattering equations. The remainder uses explicit time discretizations. More specifically, we use implicit and explicit Euler time-discretizations in this work. Let us start with $\mathbf{X}^{n}_{\ell} = \mathbf{X}^{0}_{\ell}$, $\mathbf{S}^{n}_{\ell} = \mathbf{S}^{0}_{\ell}$ and $\mathbf{W}^{n}_{\ell} = \mathbf{W}^{0}_{\ell}$, where $\ell$ denotes the individual collision steps, i.e., $\ell\in\{1,\cdots,L-1,L,c\}$. The streaming update is the same for all collision steps. Hence, when omitting a specific collision index, we obtain
\begin{subequations}\label{eq:DLRAStreamingEuler}
\begin{alignat}{2}
\mathbf{K}^{\inter} =& \ \mathbf{K}^0 +\Delta t \left(\mathbf L_x^{(2)} \mathbf{K}^0\mathbf{\widehat A}_{x}^0 + \mathbf L_y^{(2)} \mathbf{K}^0\mathbf{\widehat A}_{y}^0 + \mathbf L_x^{(1)} \mathbf{K}^0\vert\mathbf{\widehat A}_{x}\vert^0+ \mathbf L_y^{(1)} \mathbf{K}^0\vert\mathbf{\widehat A}_{y}\vert^0\right)\;, \qquad && \mathbf{X}^{\inter}\mathbf{R}_1 = \mathbf{K}^1\;, \\
\mathbf{L}^{\inter} =& \ \mathbf{L}^0 +\Delta t \left(\mathbf{\widehat L}_x^{(2),0} \mathbf{L}^0 \mathbf A_{x}^T + \mathbf{\widehat L}_y^{(2),0}\mathbf{L}^0 \mathbf A_{y}^T + \mathbf{\widehat L}_x^{(1),0} \mathbf{L}^0 \vert\mathbf A_{x}\vert^T+ \mathbf{\widehat L}_y^{(1),0}\mathbf{L}^0 \vert\mathbf A_{y}\vert^T\right)\;, \qquad && \mathbf{W}^{\inter}\mathbf{R}_2 = \mathbf{L}^{1,T}\;, \\
\mathbf{S}^{\inter} =& \ \mathbf{\widetilde S}^0 +\Delta t \left( \mathbf{\widehat L}_x^{(2),\inter} \mathbf{\widetilde S}^0 \mathbf{\widehat A}_{x}^{\inter} + \mathbf{\widehat L}_y^{(2),\inter}\mathbf{\widetilde S}^0 \mathbf{\widehat A}_{y}^{\inter} + \mathbf{\widehat L}_x^{(1),\inter} \mathbf{\widetilde S}^0 \vert\mathbf{\widehat A}_{x}\vert^{\inter}+ \mathbf{\widehat L}_y^{(1),\inter}\mathbf{\widetilde S}^0 \vert\mathbf{\widehat A}_{y}\vert^{\inter} \right)\;, 
\end{alignat}
\end{subequations}
where $\mathbf{\widetilde S}^0 = \mathbf{X}^{\inter,T}\mathbf{X}^0\mathbf{S}^0\mathbf{W}^{0,T}\mathbf{W}^{\inter}$ and flux matrices are computed before evaluating the right-hand side. The collision equations differ for the $\ell=1$ and $\ell\in\{2,\cdots, L\}$ collided fluxes as well as the collided flux. For $\ell=1$, we have
\begin{subequations}\label{eq:DLRAScatteringl1Euler}
\begin{alignat}{2}
    \mathbf{K}_1^1 =& \ \frac1{1+\Delta t \Sigma(t_1)}\left(\mathbf{K}_1^{\inter} + \Delta t\bm{\psi}(t_1)\mathbf{T}_{M}\bm{\Sigma}(t_1)\mathbf{W}_1^{\inter}\right)\;, \qquad && \mathbf{X}_1^1\mathbf{R}_1 = \mathbf{K}^1_1\;, \\
    \mathbf{L}_1^1 =& \ \frac1{1+\Delta t \Sigma(t_1)}\left(\mathbf{L}_1^{\inter} + \Delta t\mathbf{X}_1^{\inter,T}\bm{\psi}(t_1)\mathbf{T}_{M}\bm{\Sigma}(t_1)\right)\;, \qquad && \mathbf{W}_1^1\mathbf{R}_2 = \mathbf{L}_1^{1,T}\;, \\
    \mathbf{S}_1^1 =& \ \frac1{1+\Delta t \Sigma(t_1)}\left(\mathbf{\widetilde S}_1^{\inter} + \Delta t\mathbf{X}_1^{1,T}\bm{\psi}(t_1)\mathbf{T}_{M}\bm{\Sigma}(t_1)\mathbf{W}_1^1\right).
\end{alignat}
\end{subequations}
For $\ell\in\{2,\cdots, L\}$ we have
\begin{subequations}\label{eq:DLRAScatteringlEuler}
\begin{alignat}{2}
    \mathbf{K}_{\ell}^1 =& \ \frac1{1+\Delta t \Sigma(t_{1})}\left(\mathbf{K}_{\ell}^{\inter} + \Delta t\mathbf{X}_{\ell-1}^{1}\mathbf{S}_{\ell-1}^1\mathbf{W}_{\ell-1}^{1,T}\bm{\Sigma}(t_1)\mathbf{W}_{\ell}^{\inter}\right)\;, \qquad && \mathbf{X}_{\ell}^1\mathbf{R}_1 = \mathbf{K}_{\ell}^1\;, \\
    \mathbf{L}_{\ell}^1 =& \ \frac1{1+\Delta t \Sigma(t_1)}\left(\mathbf{L}_{\ell}^{\inter} + \Delta t\mathbf{X}_{\ell}^{\inter,T}\mathbf{X}_{\ell-1}^{1}\mathbf{S}_{\ell-1}^1\mathbf{W}_{\ell-1}^{1,T}\bm{\Sigma}(t_1)\right)\;, \qquad && \mathbf{W}_{\ell}^1\mathbf{R}_2 = \mathbf{L}_{\ell}^{1,T}\;, \\
    \mathbf{S}_{\ell}^1 =& \ \frac1{1+\Delta t \Sigma(t_1)}\left(\mathbf{\widetilde S}_{\ell}^{\inter} + \Delta t\mathbf{X}_{\ell}^{1,T}\mathbf{X}_{\ell-1}^{1}\mathbf{S}_{\ell-1}^1\mathbf{W}_{\ell-1}^{1,T}\bm{\Sigma}(t_1)\mathbf{W}_{\ell}^1\right).\label{eq:DLRAScatteringlEulerS}
\end{alignat}
\end{subequations}
The collided flux is then updated through
\begin{subequations}\label{eq:DLRACollisionCEuler}
\begin{alignat}{2}
    \mathbf{K}_{c}^1 =& \ \mathbf{K}_{c}^{\inter} + \Delta t\mathbf{X}_{L}^{1}\mathbf{S}_{L}^1\mathbf{W}_{L}^{1,T}\bm{\Sigma}(t_1)\mathbf{W}_{c}^{\inter}\;, \qquad && \mathbf{X}_{c}^1\mathbf{R}_1 = \mathbf{K}^{1}_c\;, \\
    \mathbf{L}_{c}^1 =& \ \mathbf{L}_{c}^{\inter} + \Delta t\mathbf{X}_{c}^{\inter,T}\mathbf{X}_{L}^{1}\mathbf{S}_{L}^1\mathbf{W}_{L}^{1,T}\bm{\Sigma}(t_1)\;, \qquad && \mathbf{\widetilde W}_{c}^1\mathbf{R}_2 = \mathbf{L}_{c}^{1,T}\;, \\
    \mathbf{\widebar S}_{c}^1 =& \ \mathbf{\widetilde S}_{c}^{\inter} + \Delta t\mathbf{X}_{c}^{1,T}\mathbf{X}_{L}^{1}\mathbf{S}_{L}^1\mathbf{W}_{L}^{1,T}\bm{\Sigma}(t_1)\mathbf{\widetilde W}_{c}^1,\\
    \mathbf{\widetilde L}_{c}^1 =& \ \mathbf{\widebar S}_{c}^1\mathbf{\widetilde W}_{c}^{1,T} (\mathbf I+\Sigma_t\Delta t\mathbf I - \Delta t\bm{\Sigma})^{-1}\;, \qquad && \mathbf{W}_{c}^1\mathbf{S}_c^{1,T} = \mathbf{\widetilde L}_{c}^{1,T}\;.\label{eq:SClast}
\end{alignat}
\end{subequations}
Note that since $\bm{\Sigma}$ is a diagonal matrix, the inversion in \eqref{eq:SClast} is given explicitly without having to solve a linear system of equations. The time updated solution is then given by $\mathbf{X}^{n+1}_{\ell} = \mathbf{X}^{1}_{\ell}$, $\mathbf{S}^{n+1}_{\ell} = \mathbf{S}^{1}_{\ell}$ and $\mathbf{W}^{n+1}_{\ell} = \mathbf{W}^{1}_{\ell}$, where $\ell\in\{1,\cdots,L-1,L,c\}$.

\begin{remark}
The proposed idea of multilevel DLRA can be applied in various settings with various strategies. The core ingredient is to write the solution as a sum of different contributions
\begin{align*}
    \mathbf{u}(t) = \mathbf{u}_1(t) + \mathbf{u}_2(t) + \cdots + \mathbf{u}_L(t)\;.
\end{align*}
Strategies to write the solution as a sum of different components can be the use of telescoping identities, a split into symmetric and anti-symmetric solution contributions, a splitting of the original phase space, e.g. particles that move forward and backward and many more. In a second step, evolution equations for every component need to be derived. Third, every component $\mathbf{u}_i$ is represented through a low-rank factorization and evolution equations for every factor are derived with DLRA.
\end{remark}

\section{L$^2$-stability of the proposed scheme}\label{sec:L2Stability}
The derived method is robust in that its time step restriction (or CFL number) does not depend on small singular values of the coefficient matrix or  stiff terms arising in the scattering step. By the choice of the splitting steps, we ensure that this stability is achieved without having to invert matrices or having to solve a nonlinear problem, which for implicit time integration methods is commonly the case. To determine a suitable CFL condition, let us investigate the L$^2$-stability of the proposed scheme, which follows the approach taken in \cite{kusch2021stability}. In contrast to \cite{kusch2021stability}, the update equations include the inverse density, which will pose a challenge. We first note that the inverse density can be pulled out of the stencil matrices. For this, we define the sparse diffusion stencil matrices without density dependence $\mathbf{T}_{x,y}^{(1)}\in \mathbb{R}^{n_x \times n_x}$ as
\begin{align*}
&{T}^{(1)}_{x,\text{idx}(i,j),\text{idx}(i,j)}= \frac{1}{\Delta x}\;, \quad T^{(1)}_{x,\text{idx}(i,j),\text{idx}(i\pm 1,j)}= -\frac{1}{2\Delta x}\;, \\
&{T}^{(1)}_{y,\text{idx}(i,j),\text{idx}(i,j)}= \frac{1}{\Delta y}\;, \quad T^{(1)}_{y,\text{idx}(i,j),\text{idx}(i,j\pm 1)}= -\frac{1}{2\Delta y}\;,
\end{align*}
as well as the sparse advection stencil matrices without density dependence $\mathbf{T}_{x,y}^{(2)}\in \mathbb{R}^{n_x \times n_x}$ as
\begin{align*}
{T}^{(2)}_{x,\text{idx}(i,j),\text{idx}(i\pm 1,j)}= \pm \frac{1}{2\Delta x}\;,\quad T^{(2)}_{y,\text{idx}(i,j),\text{idx}(i,j\pm 1)}= \pm \frac{1}{2\Delta y}\;.
\end{align*}
With $\bm{\rho}^{-1} = \text{diag}\left((\rho_{\text{idx}(i,j)})_{i,j = 1}^{N_x,N_y}\right)$, we have that $\mathbf{L}_{x,y}^{(1,2)}=\mathbf{T}_{x,y}^{(1,2)}\bm{\rho}^{-1}$.
Following \cite{kusch2021stability}, we pursue a von Neumann-like approach and investigate how the scheme amplifies and dampens certain Fourier modes. Let us store these modes in a matrix $\mathbf E\in\mathbb{C}^{n_x\times n_x}$ with entries
\begin{align*}
E_{\text{idx}(\ell, k), \text{idx}(\alpha, \beta)} = \sqrt{\Delta x\Delta y}\exp(i\alpha\pi x_{\ell})\exp(i\beta\pi y_{k})\;,
\end{align*}
where $i\in\mathbb{C}$ denotes the imaginary unit. This matrix has several properties. First, it is orthonormal, i.e., $\mathbf E\mathbf E^H = \mathbf E^H\mathbf E = \mathbf I$, where an uppercase $H$ denotes the complex transpose. Second, the matrix $\bm{E}$ applied to the spatial flux matrices diagonalizes the scheme:
\begin{align}\label{eq:diagScheme}
\mathbf{T}_{x,y}^{(1,2)}\mathbf E = \mathbf E\mathbf D_{x,y}^{(1,2)}\;.
\end{align}
The diagonal matrices $\mathbf D_{x,y}^{(1,2)}\in\mathbb{R}^{n_x\times n_x}$ have entries
\begin{align*}
    D_{x,\text{idx}(\alpha, \beta),\text{idx}(\alpha', \beta')}^{(1)} =& \frac{1}{2\Delta x} \left(e^{i\alpha\pi \Delta x}-2+e^{-i\alpha\pi \Delta x}\right)\delta_{\alpha\alpha'}\delta_{\beta\beta'} = \frac{1}{\Delta x} \left(\cos(\alpha\pi \Delta x)-1\right)\delta_{\alpha\alpha'}\delta_{\beta\beta'}\;,\\ 
        D_{y,\text{idx}(\alpha, \beta),\text{idx}(\alpha', \beta')}^{(1)} =& \frac{1}{2\Delta y} \left(e^{i\alpha\pi \Delta y}-2+e^{-i\alpha\pi \Delta y}\right)\delta_{\alpha\alpha'}\delta_{\beta\beta'}= \frac{1}{\Delta y} \left(\cos(\alpha\pi \Delta y)-1\right)\delta_{\alpha\alpha'}\delta_{\beta\beta'}\;,\\
            D_{x,\text{idx}(\alpha, \beta),\text{idx}(\alpha', \beta')}^{(2)} =& \frac{1}{2\Delta x}(e^{i\alpha\pi \Delta x}-e^{-i\alpha\pi \Delta x})\delta_{\alpha\alpha'}\delta_{\beta\beta'} = -\frac{i}{\Delta x} \sin(\alpha\pi \Delta x)\delta_{\alpha\alpha'}\delta_{\beta\beta'}\;,\\
    D_{y,\text{idx}(\alpha, \beta),\text{idx}(\alpha', \beta')}^{(2)} =& \frac{1}{2\Delta y}(e^{i\alpha\pi \Delta y}-e^{-i\alpha\pi \Delta y})\delta_{\alpha\alpha'}\delta_{\beta\beta'}= -\frac{i}{\Delta y} \sin(\alpha\pi \Delta y)\delta_{\alpha\alpha'}\delta_{\beta\beta'}\;.
\end{align*}
With these tools at hand, we can prove stability. Let us start with the streaming steps:
\begin{theorem}\label{th:firstDRL}
Assume that the CFL condition
\[ \frac{\lambda_{max}(\mathbf{A}_{x,y})}{\rho_{\text{min}}}\frac{\Delta t}{\Delta x} \leq \frac12 \] 
holds true. Then, the streaming scheme \eqref{eq:DLRAStreamingEuler} is L$^2$-stable, i.e., 
\begin{align}\label{eq:CFL}
  \Vert  \mathbf{X}^{\inter}\mathbf{S}^{\inter}\mathbf{W}^{\inter,T} \Vert_F \leq  \Vert \mathbf{X}^0\mathbf{S}^{0}\mathbf{W}^{0,T} \Vert_F\;.
\end{align}
\end{theorem} 
\begin{proof}
First, we include the identity $\mathbf E\mathbf E^H$ inside the spatial flux matrices of the $S$-step:
\begin{align*}
\mathbf{\widehat L}_{x,y}^{(1),\inter}=\mathbf{X}^{\inter,T}\mathbf{T}_{x,y}^{(1)}\mathbf E\mathbf E^H\bm{\rho}^{-1} \mathbf{X}^{\inter}\stackrel{\eqref{eq:diagScheme}}{=}\mathbf{X}^{\inter,T}\mathbf E\mathbf{D}_{x,y}^{(1)}\mathbf E^H\bm{\rho}^{-1} \mathbf{X}^{\inter}\;,\\
\mathbf{\widehat L}_{x,y}^{(2),\inter}= \mathbf{X}^{\inter,T}\mathbf{T}_{x,y}^{(2)}\mathbf E\mathbf E^H\bm{\rho}^{-1} \mathbf{X}^{\inter} \stackrel{\eqref{eq:diagScheme}}{=} \mathbf{X}^{\inter,T}\mathbf E\mathbf{D}_{x,y}^{(2)}\mathbf E^H\bm{\rho}^{-1} \mathbf{X}^{\inter}\;.
\end{align*}
Then, with $\bm{\widetilde \rho}^{-1}:= \mathbf E^H\bm{\rho}^{-1} \mathbf E$ and $\mathbf{u} = \mathbf X^{\inter} \mathbf{\widetilde S}^0 \mathbf W^{\inter,T}$, we have
\begin{alignat*}{2}
\mathbf{S}^{\inter} =& \ \mathbf{\widetilde S}^0 &&+ \Delta t \left( \mathbf{\widehat L}_x^{(2),\inter} \mathbf{\widetilde S}^0 \mathbf{\widehat A}_{x}^{\inter} + \mathbf{\widehat L}_y^{(2),\inter}\mathbf{\widetilde S}^0 \mathbf{\widehat A}_{y}^{\inter} + \mathbf{\widehat L}_x^{(1),\inter} \mathbf{\widetilde S}^0 \vert\mathbf{\widehat A}_{x}\vert^{\inter}+ \mathbf{\widehat L}_y^{(1),\inter}\mathbf{\widetilde S}^0 \vert\mathbf{\widehat A}_{y}\vert^{\inter} \right)\\
=& \ \mathbf{X}^{\inter,T}\mathbf E\mathbf E^H\bm{u}\mathbf W^{\inter} &&+ \Delta t \mathbf{X}^{\inter,T}\mathbf E\left( \mathbf{D}_{x}^{(2)} \bm{\widetilde \rho}^{-1}\mathbf E^H\bm{u} \mathbf{A}_{x} + \mathbf{D}_{x}^{(1)} \bm{\widetilde \rho}^{-1}\mathbf E^H\bm{u} \vert\mathbf{A}_{x}\vert\right)\mathbf W^{\inter}\\
& &&+\Delta t \mathbf{X}^{\inter,T}\mathbf E\left(\mathbf{D}_{y}^{(2)}\bm{\widetilde \rho}^{-1}\mathbf E^H\bm{u} \mathbf{A}_{y}+ \mathbf{D}_{y}^{(1)}\bm{\widetilde \rho}^{-1}\mathbf E^H\bm{u} \vert\mathbf{A}_{y}\vert \right)\mathbf W^{\inter}\;.
\end{alignat*}
Taking the norm and noting that for the spectral norm we have $\Vert \mathbf{X}^{\inter,T}\mathbf E \Vert = \Vert \mathbf W^{\inter} \Vert = 1$ yields
\begin{subequations}\label{eq:proof1normS}
\begin{align}
    \Vert \mathbf{S}^{\inter} \Vert_F \leq& \left\Vert \mathbf E^H\bm{u} +\Delta t \left(\mathbf{D}_{x}^{(2)} \bm{\widetilde \rho}^{-1}\mathbf E^H\bm{u} \mathbf{A}_{x} + \mathbf{D}_{x}^{(1)} \bm{\widetilde \rho}^{-1}\mathbf E^H\bm{u} \vert\mathbf{A}_{x}\vert +\mathbf{D}_{y}^{(2)}\bm{\widetilde \rho}^{-1}\mathbf E^H\bm{u} \mathbf{A}_{y}+ \mathbf{D}_{y}^{(1)}\bm{\widetilde \rho}^{-1}\mathbf E^H\bm{u} \vert\mathbf{A}_{y}\vert \right) \right\Vert_F\nonumber\\
    \leq&\left\Vert \frac12\mathbf E^H\bm{u} +\Delta t \left(\mathbf{D}_{x}^{(2)} \bm{\widetilde \rho}^{-1}\mathbf E^H\bm{u} \mathbf{A}_{x}+\mathbf{D}_{x}^{(1)} \bm{\widetilde \rho}^{-1}\mathbf E^H\bm{u} \vert\mathbf{A}_{x}\vert\right)\right\Vert_F\label{eq:proof1a}\\
    &+ \left\Vert\frac12\mathbf E^H\bm{u} +\Delta t\left(\mathbf{D}_{y}^{(2)}\bm{\widetilde \rho}^{-1}\mathbf E^H\bm{u} \mathbf{A}_{y}+ \mathbf{D}_{y}^{(1)}\bm{\widetilde \rho}^{-1}\mathbf E^H\bm{u} \vert\mathbf{A}_{y}\vert \right) \right\Vert_F\;.\label{eq:proof1b}
\end{align}
\end{subequations}
Now we investigate the terms \eqref{eq:proof1a} and \eqref{eq:proof1b} individually. Recall that we have $\mathbf A_{x,y} = \mathbf V_{x,y} \mathbf\Lambda_{x,y} \mathbf V_{x,y}^{T}$. Then for \eqref{eq:proof1a}, which we denote by $\Vert \mathbf{e}\Vert_F$, we define $\mathbf w := \mathbf E^H\bm{u} \mathbf V_{x}$ which gives
\begin{align*}
    \Vert \mathbf{e}\Vert_F =& \ \left\Vert \frac12\mathbf{w}\mathbf V_{x}^T +\Delta t \left(\mathbf{D}_{x}^{(2)} \bm{\widetilde \rho}^{-1}\mathbf{w}\mathbf{\Lambda}_{x}+\mathbf{D}_{x}^{(1)} \bm{\widetilde \rho}^{-1}\mathbf{w} \vert\mathbf{\Lambda}_{x}\vert\right)\mathbf V_{x}^T\right\Vert_F\\
    \leq& \ \left\Vert \frac12\mathbf{w} +\Delta t \left(\mathbf{D}_{x}^{(2)} \bm{\widetilde \rho}^{-1}\mathbf{w} \mathbf{\Lambda}_{x}+\mathbf{D}_{x}^{(1)} \bm{\widetilde \rho}^{-1}\mathbf{w} \vert\mathbf{\Lambda}_{x}\vert\right)\right\Vert_F\;.
\end{align*}
Note that with $\mathbf{e}_k := (e_{jk})_{j=1}^{n_x}$ and $\mathbf{w}_k := (w_{jk})_{j=1}^{n_x}$, we have
\begin{align*}
    \Vert \mathbf{e}\Vert_F^2 =& \sum_{k = 1}^m \Vert \mathbf{e}_k\Vert_2^2 = \sum_{k = 1}^m\left\Vert \frac12\mathbf{w}_k +\Delta t \left(\mathbf{D}_{x}^{(2)} \bm{\widetilde \rho}^{-1}\mathbf{w}_k \lambda_{k}^x+\mathbf{D}_{x}^{(1)} \bm{\widetilde \rho}^{-1}\mathbf{w}_k \vert\lambda_{k}^x\vert\right)\right\Vert_2^2 \\
    \leq & \sum_{k = 1}^m\left\Vert \frac12\mathbf{I} +\Delta t \left(\lambda_{k}^x\mathbf{D}_{x}^{(2)} +\vert\lambda_{k}^x\vert\cdot\mathbf{D}_{x}^{(1)} \right)\bm{\widetilde \rho}^{-1}\right\Vert^2\cdot \left\Vert\mathbf{w}_k\right\Vert_2^2\;,
\end{align*}
where $\Vert\cdot\Vert_2$ denotes the Euclidean norm. Now, for the spectral norm in the above expression we have the upper bound
\begin{align*}
    \left\Vert \frac12\mathbf{I} +\Delta t \left(\lambda_{k}^x\mathbf{D}_{x}^{(2)} +\vert\lambda_{k}^x\vert\cdot\mathbf{D}_{x}^{(1)} \right)\bm{\widetilde \rho}^{-1}\right\Vert \leq \left\vert \frac12 + \Delta t \lambda_{\text{max}}\left( \lambda_{k}^x\mathbf{D}_{x}^{(2)} +\vert\lambda_{k}^x\vert\cdot\mathbf{D}_{x}^{(1)}\right)\cdot \lambda_{\text{max}}(\bm{\widetilde \rho}^{-1}) \right\vert\;.
\end{align*}
We note that $\lambda_{\text{max}}(\bm{\widetilde \rho}^{-1}) = \min_j \rho_j^{-1} =:\rho_{\text{min}}^{-1}$. With $\nu := \frac{\max_k \vert\lambda_{k}^x\vert \Delta t}{\rho_{\text{min}}\Delta x}$ we have
\begin{align*}
    \Vert \mathbf{e}\Vert_F \leq& \max_{\alpha} \left\vert \frac12 + \nu (\cos(\alpha\pi\Delta x)-1) - i\nu \sin(\alpha\pi\Delta x) \right\vert \cdot \left\Vert\mathbf{w}\right\Vert\\
    =& \max_{\alpha} \sqrt{ \frac14 + \nu (\cos(\alpha\pi\Delta x)-1) + \nu^2 (\cos(\alpha\pi\Delta x)-1)^2 + \nu^2 \sin^2(\alpha\pi\Delta x) } \cdot \left\Vert\mathbf{w}\right\Vert\\
    =& \max_{\alpha} \sqrt{ \frac14 + \nu(1- 2\nu )(\cos(\alpha\pi\Delta x)-1)  } \cdot \left\Vert\mathbf{w}\right\Vert\;.
\end{align*}
To obtain stability, we need $\Vert \mathbf{e}\Vert_F \leq \left\Vert\mathbf{w}\right\Vert/2$, i.e., $\nu \leq \frac12$. In the same way, we can get an estimate for \eqref{eq:proof1b}, which yields $\frac{\max_k \vert\lambda_{k}^y\vert \Delta t}{\rho_{\text{min}}\Delta y}\leq \frac12$. With
\begin{align*}
    \Vert \mathbf{w}\Vert_F = \Vert \mathbf E^H\mathbf{u} \mathbf V_{x,y} \Vert_F\leq \Vert \mathbf{u} \Vert_F = \Vert \mathbf X^{\inter} \mathbf{\widetilde S}^0 \mathbf W^{\inter,T} \Vert_F \leq \Vert \mathbf{S}^0 \Vert_F\;,
\end{align*}
we know that $\Vert \mathbf{S}^{\inter} \Vert_F \leq \Vert\mathbf{S}^0\Vert_F$, which with Parseval's identity proves the theorem.
\end{proof}
\begin{theorem}\label{th:scattering}
Assume that the CFL conditions \eqref{eq:CFL} and
\begin{align}\label{eq:CFLScatter}
    \max_k\frac{1}{1+\Delta t\Sigma_t(t)-\Delta t\Sigma_{kk}(t)} \leq 1
\end{align}
hold true for all pseudo-times $t\in [0,T]$. Then, the scheme is L$^2$-stable, in the sense that with $\mathbf{u}_c^1:=\mathbf{X}_c^1\mathbf{S}_c^1\mathbf{W}_c^{1,T}$ and $\mathbf{u}_{\ell}^1:=\mathbf{X}_{\ell}^1\mathbf{S}_{\ell}^1\mathbf{W}_{\ell}^{1,T}$ we have
\begin{align*}
  \Vert\mathbf{u}_c^1\Vert_F+\sum_{\ell=1}^L\Vert\mathbf{u}_{\ell}^1\Vert_F+\Vert\bm{\psi}(t_1)\Vert_F \leq \Vert\mathbf{u}_c^0\Vert_F+\sum_{\ell=1}^L\Vert\mathbf{u}_{\ell}^0\Vert_F+\Vert\bm{\psi}(t_0)\Vert_F\;.
\end{align*}
\end{theorem}
\begin{proof}
We start with the collided equations. The last two steps of \eqref{eq:DLRACollisionCEuler} read
\begin{align*}
    \mathbf{\widebar S}_{c}^1 =& \ \mathbf{\widetilde S}_{c}^{\inter} + \Delta t\mathbf{X}_{c}^{1,T}\mathbf{X}_{L}^{1}\mathbf{S}_{L}^1\mathbf{W}_{L}^{1,T}\bm{\Sigma}(t_1)\mathbf{\widetilde W}_{c}^1\;,\\
    \mathbf{\widetilde L}_{c}^1 =& \ \mathbf{\widebar S}_{c}^1\mathbf{\widetilde W}_{c}^{1,T} (\mathbf I+\Sigma_t\Delta t\mathbf I - \Delta t\bm{\Sigma})^{-1}\;.
\end{align*}
Writing this as a single expression gives
\begin{align*}
    \mathbf{\widetilde L}_{c}^1 = \ \left(\mathbf{\widetilde S}_{c}^{\inter} + \Delta t\mathbf{X}_{c}^{1,T}\mathbf{X}_{L}^{1}\mathbf{S}_{L}^1\mathbf{W}_{L}^{1,T}\bm{\Sigma}(t_1)\mathbf{\widetilde W}_{c}^1\right)\mathbf{\widetilde W}_{c}^{1,T} (\mathbf I+\Sigma_t\Delta t\mathbf I - \Delta t\bm{\Sigma})^{-1}\;.
\end{align*}
We take the norm of the above expression and note that since $\mathbf{\widetilde L}_{c}^1 = \mathbf{S}_{c}^1\mathbf{W}_{c}^{1,T}$
\begin{align*}
    \Vert\mathbf{S}_{c}^1\Vert_F \leq \ \max_k\frac{1}{1+\Delta t\Sigma_t-\Delta t\Sigma_{kk}}\Vert\mathbf{\widetilde S}_{c}^{\inter}\Vert_F + \max_k\frac{\Delta t \Sigma_{kk}}{1+\Delta t\Sigma_t-\Delta t\Sigma_{kk}}\Vert\mathbf{S}_{L}^1\Vert_F\;.
\end{align*}
Note that we used
\begin{align*}
    \Vert \bm\Sigma \Vert \cdot \Vert(\mathbf I+\Sigma_t\Delta t\mathbf I - \Delta t\bm{\Sigma})^{-1}\Vert =\max_k\Sigma_{kk}\cdot \max_k\frac{1}{1+\Delta t\Sigma_t-\Delta t\Sigma_{kk}} = \max_k\frac{\Sigma_{kk}}{1+\Delta t\Sigma_t-\Delta t\Sigma_{kk}}\;.
\end{align*}
Adding $\Vert\mathbf{S}_{L}^1\Vert_F$ to both sides and noting that by Theorem~\ref{th:firstDRL}, we have $\Vert\mathbf{\widetilde S}_{c}^{\inter}\Vert_F \leq\Vert\mathbf{S}_{c}^{\inter}\Vert_F\leq \Vert\mathbf{S}_{c}^0\Vert_F$ leads to
\begin{align*}
    \Vert\mathbf{S}_{c}^1\Vert_F+\Vert\mathbf{S}_{L}^1\Vert_F \leq \ \max_k\frac{1}{1+\Delta t\Sigma_t-\Delta t\Sigma_{kk}}\Vert\mathbf{S}_{c}^0\Vert_F + \max_k\frac{1+\Delta t\Sigma_t}{1+\Delta t\Sigma_t-\Delta t\Sigma_{kk}}\Vert\mathbf{S}_{L}^1\Vert_F\;.
\end{align*}
By \eqref{eq:DLRAScatteringlEulerS} we have
\begin{align*}
    \Vert\mathbf{S}_{L}^1\Vert_F \leq& \frac1{1+\Delta t \Sigma_t}\left(\Vert\mathbf{\widetilde S}_{L}^{\inter}\Vert_F + \Delta t\max_k\Sigma_{kk}\Vert\mathbf{S}_{L-1}^1\Vert_F\right)\\
    \leq& \frac1{1+\Delta t \Sigma_t}\left(\Vert\mathbf{S}_{L}^{0}\Vert_F + \Delta t\max_k\Sigma_{kk}\Vert\mathbf{S}_{L-1}^1\Vert_F\right)\;.
\end{align*}
Using the chosen CFL condition gives
\begin{align*}
    \Vert\mathbf{S}_{c}^1\Vert_F+\Vert\mathbf{S}_{L}^1\Vert_F \leq \ \Vert\mathbf{S}_{c}^0\Vert_F + \Vert\mathbf{S}_{L}^0\Vert_F + \max_k\frac{\Delta t \Sigma_{kk}}{1+\Delta t\Sigma_t-\Delta t\Sigma_{kk}}\Vert\mathbf{S}_{L-1}^1\Vert_F\;.
\end{align*}
Recursively continuing this process until $\ell = 1$ gives
\begin{align*}
    \Vert\mathbf{S}_{c}^1\Vert_F+\sum_{\ell=1}^L\Vert\mathbf{S}_{\ell}^1\Vert_F \leq \ \Vert\mathbf{S}_{c}^0\Vert_F + \sum_{\ell=1}^L\Vert\mathbf{S}_{\ell}^0\Vert_F + \max_k\frac{\Delta t \Sigma_{kk}}{1+\Delta t\Sigma_t-\Delta t\Sigma_{kk}}\Vert\bm{\psi}(t_1)\Vert_F\;.
\end{align*}
Adding $\Vert\bm{\psi}(t_1)\Vert_F$ on both sides, noting that (when choosing an implicit time discretization for scattering of uncollided particles)
\begin{align*}
    \Vert\bm{\psi}(t_1)\Vert_F = \frac{1}{1+\Delta t\Sigma_t}\Vert\bm{\psi}(t_0)\Vert_F
\end{align*}
and the use of Parseval's identity proves the theorem.
\end{proof}
\begin{remark}
Commonly, the CFL condition of the scattering step is fulfilled automatically, since $\Sigma_{kk}\leq \Sigma_t$ for all $k$. Therefore, a time step restriction is only imposed by the streaming update through \eqref{eq:CFL}.
\end{remark}
\section{Extension to rank adaptivity}\label{sec:rankAdapt}
In a last step, we discuss the extension of the proposed scheme to the rank adaptive integrator of \cite{ceruti2021rank} and how it simplifies for a forward Euler time discretization. The core ingredient of this method is to extend the time updated basis with the basis at time $t_0$. I.e., for the streaming step, the updated basis becomes $\mathbf{\widehat{X}}^{\inter} = [\mathbf{X}^0,\mathbf{\widebar{X}}^{\inter}]$, where $\mathbf{\widebar X}^{\inter}$ is chosen such that the column range of $\mathbf{\widehat{X}}^{\inter}$ contains $\mathbf{K}(t_1)$ from the streaming $K$-step and the basis is orthonormal, i.e., $\mathbf{X}^{0,T}\mathbf{\widebar X}^{\inter} = \mathbf{0}$ and $\mathbf{\widebar X}^{\inter,T}\mathbf{\widebar X}^{\inter} = \mathbf{I}$. Hence the matrix to compute the initial condition of the $S$-step reads $\mathbf{\widehat M} = \mathbf{\widehat X}^{\inter,T}\mathbf{X}^0 = [\mathbf{I}, \mathbf{0}]^T\in\mathbb{R}^{2r_0\times r_0}$, where $r_0$ is the rank at time $t_0$. In the same manner, we have $\mathbf{\widehat N} = [\mathbf{I}, \mathbf{0}]^T\in\mathbb{R}^{2r_0\times r_0}$. Hence, as pointed out in \cite{ceruti2021rank}, we have $\mathbf{\widehat X}^{{\inter}}\mathbf{\widehat S}(t_0)\mathbf{\widehat W}^{{\inter},T} = \mathbf{X}^0\mathbf{S}(t_0)\mathbf{W}^{0,T}\in\mathcal{M}_{r_0}$. The fact that the initial condition of the $S$-step is of rank $r_0$ can be used to reduce computational costs when using an explicit Euler step. The $S$-step of the rank adaptive integrator for the streaming step then reads
\begin{align*}
    \mathbf{\widehat S}^{\inter} =& \ [\mathbf{I}, \mathbf{0}]^T\ \mathbf{S}^{0} \ [\mathbf{I}, \mathbf{0}] +\Delta t \left( \mathbf{\widebar L}_x^{(2),\inter} \mathbf{S}^0 \mathbf{\widebar A}_{x}^{\inter} + \mathbf{\widebar L}_y^{(2),\inter}\mathbf{S}^0 \mathbf{\widebar A}_{y}^{\inter} + \mathbf{\widebar L}_x^{(1),\inter} \mathbf{S}^0 \vert\mathbf{\widebar A}_{x}\vert^{\inter}+ \mathbf{\widebar L}_y^{(1),\inter}\mathbf{S}^0 \vert\mathbf{\widebar A}_{y}\vert^{\inter} \right)\;,
\end{align*}
where the flux matrices are given by
\begin{alignat*}{2}
   &\mathbf{\widebar A}_{x,y}^{\inter} := \mathbf{W}^{0,T} \mathbf A_{x,y}^T\mathbf{\widehat W}^{\inter}\in\mathbb{R}^{r_0\times 2r_0}\;, \qquad &&\vert\mathbf{\widebar A}_{x,y}\vert^{\inter}:=\mathbf{W}^{0,T} \vert\mathbf A_{x,y}\vert^T\mathbf{\widehat W}^{\inter}\in\mathbb{R}^{r_0\times 2r_0}\;, \\
   &\mathbf{\widebar L}_{x,y}^{(2),\inter} := \mathbf{\widehat X}^{\inter,T}\mathbf L_{x,y}^{(2)} \mathbf{X}^0\in\mathbb{R}^{2r_0\times r_0}\;, \qquad &&\mathbf{\widebar L}_{x,y}^{(1),\inter}:=\mathbf{\widehat X}^{\inter,T}\mathbf L_{x,y}^{(1)} \mathbf{X}^0\in\mathbb{R}^{2r_0\times r_0}\;.
\end{alignat*}
Hence, the flux matrices that need to be computed have $2 r_0^2$ entries. When using more general time integration schemes to solve the $S$-step for the rank adaptive integrator, the flux matrices have $4 r_0^2$ entries. We determine the solution factors after the streaming update $\mathbf{X}^{\inter},\mathbf{S}^{\inter}$ and $\mathbf{W}^{\inter}$ as well as the rank $r_{\inter}$ through the truncation step of the rank adaptive integrator. For components $\ell \in \{1,2,\cdots,L,c\}$, the scattering steps have modified $S$-step equations 
\begin{align*}
    \mathbf{\widehat S}_1^1 =& \ \frac1{1+\Delta t \Sigma(t_1)}\left([\mathbf{I}, \mathbf{0}]^T\ \mathbf{S}_{1}^{\inter} \ [\mathbf{I}, \mathbf{0}] + \Delta t\mathbf{\widehat X}_1^{1,T}\bm{\psi}(t_1)\mathbf{T}_{M}\bm{\Sigma}(t_1)\mathbf{\widehat W}_1^1\right)\;,\\
    \mathbf{\widehat S}_{\ell}^1 =& \ \frac1{1+\Delta t \Sigma(t_1)}\left([\mathbf{I}, \mathbf{0}]^T\ \mathbf{S}_{\ell}^{\inter} \ [\mathbf{I}, \mathbf{0}] + \Delta t\mathbf{\widehat X}_{\ell}^{1,T}\mathbf{X}_{\ell-1}^{1}\mathbf{S}_{\ell-1}^1\mathbf{W}_{\ell-1}^{1,T}\bm{\Sigma}(t_1)\mathbf{\widehat W}_{\ell}^1\right)\;, \quad \text{for }\ell = 2,\cdots,L\;,\\
    \mathbf{\widehat S}_{c}^1 =& \ [\mathbf{I}, \mathbf{0}]^T\ \mathbf{S}_{c}^{\inter} \ [\mathbf{I}, \mathbf{0}] + \Delta t\mathbf{\widehat X}_{c}^{1,T}\mathbf{X}_{L}^{1}\mathbf{S}_{L}^1\mathbf{W}_{L}^{1,T}\bm{\Sigma}(t_1)\mathbf{\widehat{W}}_{c}^1\;.
\end{align*}
For $\ell \in \{1,\cdots,L,c\}$ we use $\mathbf{\widehat{X}}^1_{\ell} = [\mathbf{X}^{\inter}_{\ell},\mathbf{\widebar X}^1_{\ell}]$, where $\mathbf{\widebar X}^{1}_{\ell}$ is chosen such that the column range of $\mathbf{\widehat{X}}^{1}_{\ell}$ contains $\mathbf{K}_{\ell}(t_1)$ from the scattering $K$-step and the basis is orthonormal, i.e., $\mathbf{X}_{\ell}^{\inter,T}\mathbf{\widebar X}_{\ell}^{1} = \mathbf{0}$ and $\mathbf{\widebar X}_{\ell}^{1,T}\mathbf{\widebar X}_{\ell}^{1} = \mathbf{I}$. The directional basis $\mathbf{\widehat{W}}^1_{\ell}$ is defined analogously. Note that for the scattered particles, we need to do a final $L$-step \eqref{eq:SClast} after having updated the coefficient. In this case, the truncation step of the rank adaptive integrator yields $\mathbf{\widetilde W}_c^1$ and $\mathbf{\widebar S}_c^1$. Since \eqref{eq:SClast} is constructed through the fixed-rank projector--splitting integrator, it will not modify the rank.

Besides allowing for a dynamic choice of the rank, the rank adaptive integrator remains $L^2$ stable.
\begin{theorem}\label{th:firstDRLRankAdapt}
Assume that the CFL condition \eqref{eq:CFL} holds true. Then, the streaming scheme of the rank adaptive integrator is L$^2$-stable, i.e., 
\begin{align}\label{eq:CFLRankAdapt}
  \Vert  \mathbf{X}^{\inter}\mathbf{S}^{\inter}\mathbf{W}^{\inter,T} \Vert_F \leq  \Vert \mathbf{X}^0\mathbf{S}^{0}\mathbf{W}^{0,T} \Vert_F\;.
\end{align}
\end{theorem} 
\begin{proof}
The proof follows the steps taken in Theorems~\ref{th:firstDRL} and \ref{th:scattering}. We begin with the streaming part. Using the Fourier matrix $\mathbf E$, the spatial flux matrices can be written as
\begin{align*}
\mathbf{\widebar L}_{x,y}^{(1),\inter}=\mathbf{\widehat X}^{\inter,T}\mathbf{T}_{x,y}^{(1)}\mathbf E\mathbf E^H\bm{\rho}^{-1} \mathbf{X}^{0}=\mathbf{\widehat X}^{\inter,T}\mathbf E\mathbf{D}_{x,y}^{(1)}\mathbf E^H\bm{\rho}^{-1} \mathbf{X}^{0}\;,\\
\mathbf{\widebar L}_{x,y}^{(2),\inter}= \mathbf{\widehat X}^{\inter,T}\mathbf{T}_{x,y}^{(2)}\mathbf E\mathbf E^H\bm{\rho}^{-1} \mathbf{X}^{0} = \mathbf{\widehat X}^{\inter,T}\mathbf E\mathbf{D}_{x,y}^{(2)}\mathbf E^H\bm{\rho}^{-1} \mathbf{X}^{0}\;.
\end{align*}
With $\bm{\widetilde \rho}^{-1}:= \mathbf E^H\bm{\rho}^{-1} \mathbf E$ and $\mathbf{u}_a = \mathbf{\widehat X}^{\inter} \mathbf{\widetilde S}^0 \mathbf{\widehat W}^{\inter,T} = \mathbf{X}^0\mathbf{S}^0\mathbf{W}^{0,T}$, we have
\begin{alignat*}{2}
\mathbf{\widehat S}^{\inter} =& \ \mathbf{\widehat X}^{\inter,T}\mathbf{u}_a\mathbf{\widehat W}^{\inter} &&+\Delta t \mathbf{\widehat X}^{\inter,T}\mathbf E\left( \mathbf{D}_{x}^{(2)} \bm{\widetilde \rho}^{-1}\mathbf E^H\mathbf{u}_a \mathbf{A}_{x} + \mathbf{D}_{x}^{(1)} \bm{\widetilde \rho}^{-1}\mathbf E^H\mathbf{u}_a \vert\mathbf{A}_{x}\vert\right)\mathbf{\widehat W}^{\inter}\\
& &&+\Delta t \mathbf{\widehat X}^{\inter,T}\mathbf E\left(\mathbf{D}_{y}^{(2)}\bm{\widetilde \rho}^{-1}\mathbf E^H\mathbf{u}_a \mathbf{A}_{y}+ \mathbf{D}_{y}^{(1)}\bm{\widetilde \rho}^{-1}\mathbf E^H\mathbf{u}_a \vert\mathbf{A}_{y}\vert \right)\mathbf{\widehat W}^{\inter}\;.
\end{alignat*}
Since for the spectral norm we have $\Vert \mathbf{\widehat X}^{\inter,T}\mathbf E \Vert = \Vert \mathbf{\widehat W}^{\inter} \Vert = 1$, taking norms yields the inequality \eqref{eq:proof1normS} and the remainder of the proof follows as in Theorem~\ref{th:firstDRL}. Note that instead of $\mathbf{u} = \mathbf X^{\inter} \mathbf{\widetilde S}^0 \mathbf W^{\inter,T}$ we now use $\mathbf{u}_a$. However, this does not impact the derivation, since $\Vert\mathbf{u}_a\Vert_F = \Vert\mathbf{\widehat X}^{\inter} \mathbf{\widetilde S}^0 \mathbf{\widehat W}^{\inter,T}\Vert_F = \Vert \mathbf{S}^0 \Vert_F$. Finally, we note that the truncation step of the rank adaptive integrator does not increase the norm of the solution, in fact
\begin{align*}
    \Vert \mathbf{S}^{\inter} \Vert \leq \Vert \mathbf{\widehat S}^{\inter} \Vert \leq \Vert \mathbf{S}^0 \Vert\;.
\end{align*}
\end{proof}
Lastly, we include scattering contributions.
\begin{theorem}
Assume that the CFL conditions \eqref{eq:CFL} and \eqref{eq:CFLScatter} hold true for all pseudo-times $t\in [0,T]$. Then, the scheme is L$^2$-stable, in the sense that with $\mathbf{u}_c^1:=\mathbf{X}_c^1\mathbf{S}_c^1\mathbf{W}_c^{1,T}$ and $\mathbf{u}_{\ell}^1:=\mathbf{X}_{\ell}^1\mathbf{S}_{\ell}^1\mathbf{W}_{\ell}^{1,T}$ we have
\begin{align*}
  \Vert\mathbf{u}_c^1\Vert_F+\sum_{\ell=1}^L\Vert\mathbf{u}_{\ell}^1\Vert_F+\Vert\bm{\psi}(t_1)\Vert_F \leq \Vert\mathbf{u}_c^0\Vert_F+\sum_{\ell=1}^L\Vert\mathbf{u}_{\ell}^0\Vert_F+\Vert\bm{\psi}(t_0)\Vert_F\;.
\end{align*}
\end{theorem}
\begin{proof}
The proof is essentially that of Theorem~\ref{th:scattering}. The only difference is that we extend and truncate the basis and coefficient matrices. However, since both of these operations do not increase the Frobenius norm of the solution, the stability property is not violated by rank adaptivity. To make this observation more rigorous, we start with the collided equations. The last two steps of \eqref{eq:DLRACollisionCEuler} read
\begin{align*}
\mathbf{\widehat S}_{c}^1 =& \ [\mathbf{I}, \mathbf{0}]^T\ \mathbf{S}_{c}^{\inter} \ [\mathbf{I}, \mathbf{0}] + \Delta t\mathbf{\widehat X}_{c}^{1,T}\mathbf{X}_{L}^{1}\mathbf{S}_{L}^1\mathbf{W}_{L}^{1,T}\bm{\Sigma}(t_1)\mathbf{\widehat{W}}_{c}^1\;, \\
    \mathbf{\widetilde L}_{c}^1 =& \ \mathbf{\widebar S}_{c}^1\mathbf{\widetilde W}_{c}^{1,T} (\mathbf I+\Sigma_t\Delta t\mathbf I - \Delta t\bm{\Sigma})^{-1}\;.
\end{align*}
Note that $\Vert\mathbf{\widetilde L}_{c}^1\Vert_F = \Vert\mathbf{S}_{c}^1\Vert_F$ and $\Vert \mathbf{\widebar S}_{c}^1 \Vert_F\leq\Vert \mathbf{\widehat S}_{c}^1 \Vert_F$. Hence, taking norms of the above equations yields
\begin{align*}
\Vert\mathbf{\widehat S}_{c}^1\Vert_F \leq& \ \Vert \mathbf{S}_{c}^{\inter} \Vert_F + \Delta t\max_k\Sigma_{kk}\Vert\mathbf{S}_{L}^1\Vert_F \;, \\
    \Vert \mathbf{S}_{c}^1 \Vert_F \leq& \max_k\frac1{1+\Sigma_t\Delta t- \Delta t \Sigma_{kk}}\Vert \mathbf{\widehat S}_{c}^{1} \Vert_F = \frac1{1+\Sigma_t\Delta t- \Delta t \max_k\Sigma_{kk}}\Vert \mathbf{\widehat S}_{c}^{1} \Vert_F\;,
\end{align*}
where we used that $\max_k\vert\Sigma_{kk}\vert = \max_k\Sigma_{kk}$. Written as a single expression, we have
\begin{align}\label{eq:th4collided}
    \Vert \mathbf{S}_{c}^1 \Vert_F \leq& \max_k\frac1{1+\Sigma_t\Delta t- \Delta t \Sigma_{kk}}\left( \Vert \mathbf{S}_{c}^{\inter} \Vert_F + \Delta t\max_k\Sigma_{kk}\Vert\mathbf{S}_{L}^1\Vert_F \right) \\
    \leq& \Vert \mathbf{S}_{c}^{\inter} \Vert_F + \max_k\frac{\Delta t\Sigma_{kk}}{1+\Sigma_t\Delta t- \Delta t \Sigma_{kk}}\Vert\mathbf{S}_{L}^1\Vert_F\;.
\end{align}
For the $L$-collided particles, the coefficient vector reads
\begin{align*}
    \Vert \mathbf{S}_{L}^1 \Vert_F \leq \Vert \mathbf{\widehat S}_{L}^1 \Vert_F \leq \frac1{1+\Delta t \Sigma_t}\left(\Vert\mathbf{S}_{L}^{\inter}\Vert_F + \Delta t\max_k\Sigma_{kk}\Vert\mathbf{S}_{L-1}^1\Vert_F\right)\;.
\end{align*}
Adding $\Vert \mathbf{S}_{L}^1 \Vert_F$ to both sides of \eqref{eq:th4collided} gives
\begin{align*}
    \Vert \mathbf{S}_{c}^1 \Vert_F + \Vert \mathbf{S}_{L}^1 \Vert_F \leq& \Vert \mathbf{S}_{c}^{\inter} \Vert_F + \max_k\frac{1+\Sigma_t\Delta t}{1+\Sigma_t\Delta t- \Delta t \Sigma_{kk}}\Vert\mathbf{S}_{L}^1\Vert_F \\
    \leq& \Vert \mathbf{S}_{c}^{\inter} \Vert_F + \Vert \mathbf{S}_{L}^{\inter} \Vert_F + \max_k\frac{\Delta t \Sigma_{kk}}{1+\Sigma_t\Delta t- \Delta t \Sigma_{kk}}\Vert\mathbf{S}_{L-1}^1\Vert_F\;.
\end{align*}
By Theorem~\ref{th:firstDRLRankAdapt}, we have that 
\begin{align*}
    \Vert \mathbf{S}_{c}^1 \Vert_F + \Vert \mathbf{S}_{L}^1 \Vert_F \leq \Vert \mathbf{S}_{c}^{0} \Vert_F + \Vert \mathbf{S}_{L}^{0} \Vert_F + \max_k\frac{\Delta t \Sigma_{kk}}{1+\Sigma_t\Delta t- \Delta t \Sigma_{kk}}\Vert\mathbf{S}_{L-1}^1\Vert_F\;.
\end{align*}
Continuing recursively and following the last steps of the proof for Theorem~\ref{th:scattering} proves the statement.
\end{proof}
\section{Numerical results}\label{sec:results}
In the following, we demonstrate numerical experiments to compare conventional and the proposed methods. All results can be reproduced with the openly available code framework \cite{code}.
\subsection{Line source benchmark}\label{sec:linesource}
To demonstrate the applicability of the proposed collision source method for dynamical low-rank approximation, we first take a look at the time dependent radiation transport equation for the \textit{line source} benchmark \cite{ganapol1999homogeneous,ganapol2008analytical}
\begin{equation}
	\label{eq:rte}
	\begin{aligned}
		&\partial_t \psi + \mathbf{\Omega}\cdot\nabla_{\mathbf{x}} \psi + \Sigma_s \psi = \frac{\Sigma_s}{4\pi} \int_{\mathbb{S}^2} \psi \,d\mathbf{\Omega}\;,
		\qquad 
		(\mathbf{x}, \mathbf{\Omega}) \in [-1.5,1.5]^2 \times \mathbb{S}^2 \;, \\[1mm]
		& \psi(t_0) = \frac{1}{4\pi\sigma^2}\exp\left(-\frac{\Vert \mathbf{x} \Vert^2}{4\sigma^2} \right)\;,
	\end{aligned}
\end{equation}
where $\Sigma_s = 1$ and $\sigma = 0.03$. This equation can be recovered from the continuous slowing down approximation when choosing $\rho \equiv 1$ and treating the energy variable as time. The line source benchmark is a common test case for radiation transport problems, exposing deficiencies of different methods. A comparison of conventional methods for this benchmark can for example be found in \cite{garrett2013comparison}. Common methods require high computational costs or parameter tuning to yield a satisfactory approximation. Uses of dynamical low-rank approximation for this benchmark are \cite{PeMF20,PeM20,ceruti2021rank}, where it is observed that high ranks are needed to achieve a desired level of accuracy. Nevertheless, in comparison to classical methods, DLRA yields reduced runtimes and memory requirements. We use the following computational parameters for our calculations:
\begin{center}
    \begin{tabular}{ | l | p{8cm} |}
    \hline
    $n_x = N_x\cdot N_y=40000$ & number of spatial cells \\
    $m = (N+1)^2 = 484$ & number of spherical basis functions \\
    $n_q = 968$ & number of quadrature points for uncollided flux \\
    $t_{\text{end}}=1$ & end time \\
    \hline
    \end{tabular}
\end{center}
\begin{figure}[htp!]
    \centering
    \includegraphics[width=\linewidth]{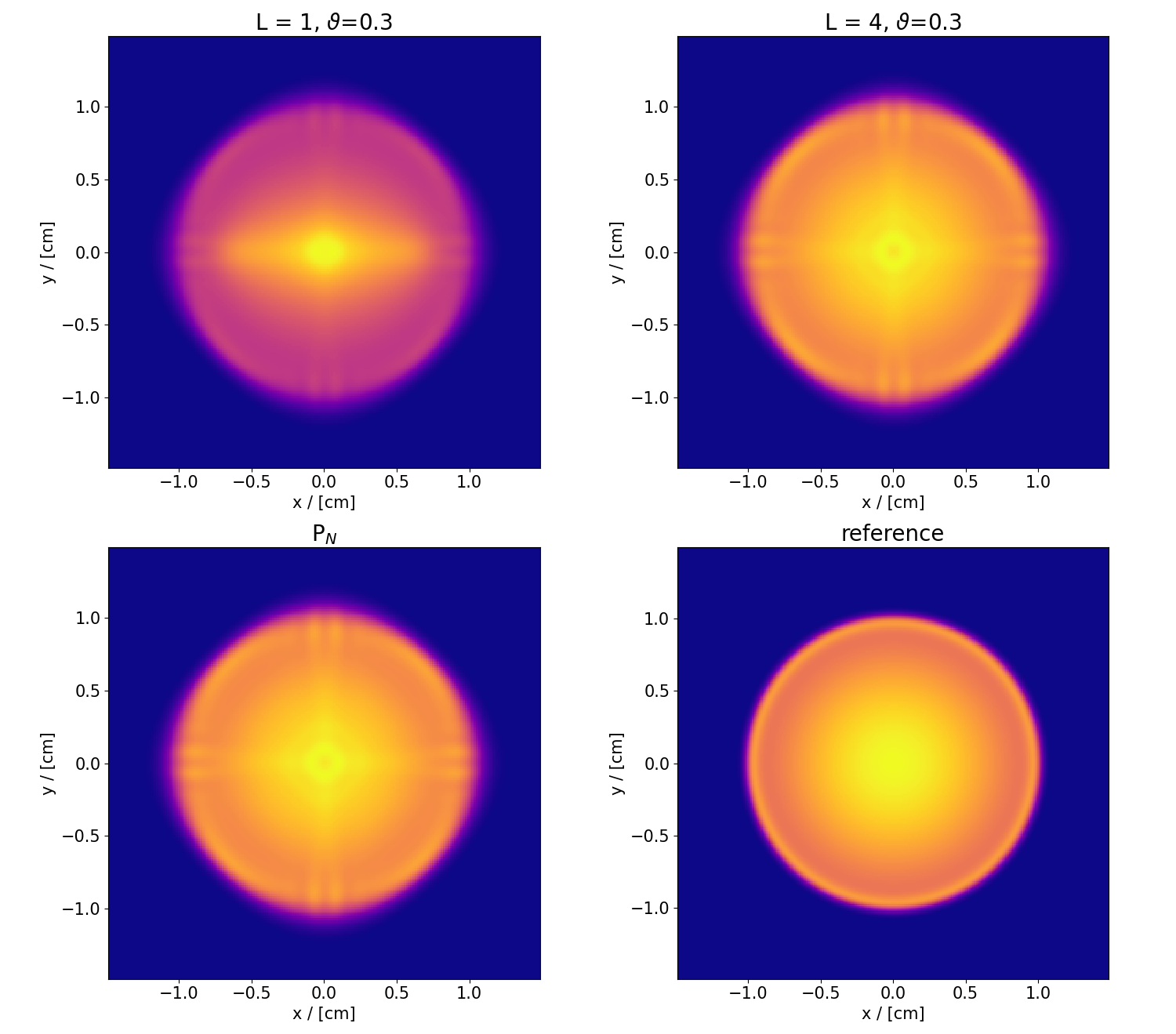}
    \caption{Scalar flux $\Phi(t=1,\mathbf{x}) = \int_{\mathbb{S}^2} \psi(t=1,\mathbf{x},\mathbf{\Omega})\,d\mathbf{\Omega} $ with different methods and analytic reference solution.}
    \label{fig:scalarFlux}
\end{figure}
The scalar flux $\Phi(t=1,\mathbf{x}) = \int_{\mathbb{S}^2} \psi(t=1,\mathbf{x},\mathbf{\Omega})\,d\mathbf{\Omega}$ computed with different methods can be found in Figure~\ref{fig:scalarFlux}. We observe a significant increase in the solution quality when using $L=4$ instead of $L=1$ levels. The level $4$ approximation with a tolerance parameter of $\vartheta = 0.3$ agrees well with the P$_N$ solution. Here, we use the term P$_N$ to indicate the use of an S$_N$ solver for uncoolided particles as well as a $P_N$ solver for the remainder. While P$_N$ takes $5408$ seconds to compute the scalar flux at time $t=1$, the DLRA methods with $L=1$ and $L=4$ levels only require $1009$ and $1278$ seconds respectively. Since particles move into all directions, a main factor in this runtime is the S$_N$ solution. Taking a look at the rank evolution in time, depicted in Figure~\ref{fig:rankInTime}, we see that most information is carried by the uncollided flux as well as solution components with little collisions.
\begin{figure}[htp!]
\centering
\begin{subfigure}{.5\textwidth}
  \centering
  \includegraphics[width=\linewidth]{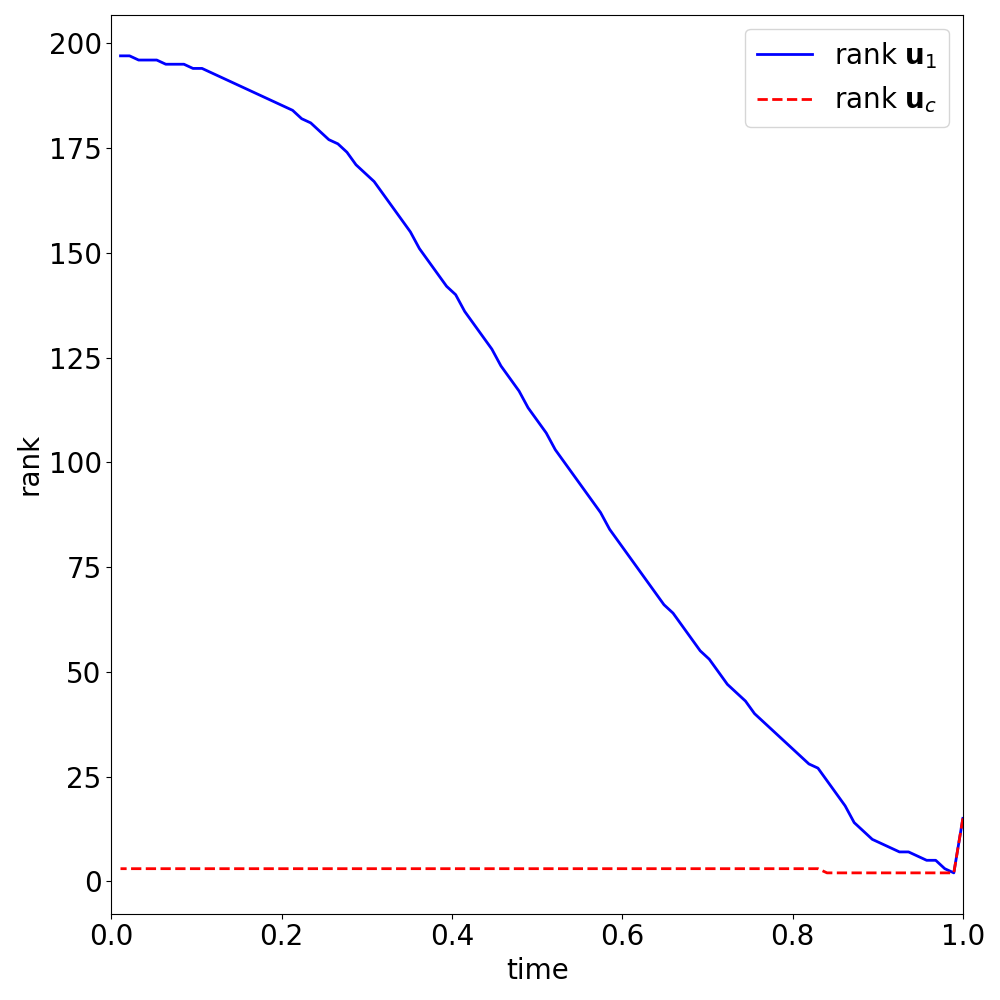}
  \caption{$\vartheta = 0.3, L = 1$}
  \label{fig:sub1}
\end{subfigure}%
\begin{subfigure}{.5\textwidth}
  \centering
  \includegraphics[width=\linewidth]{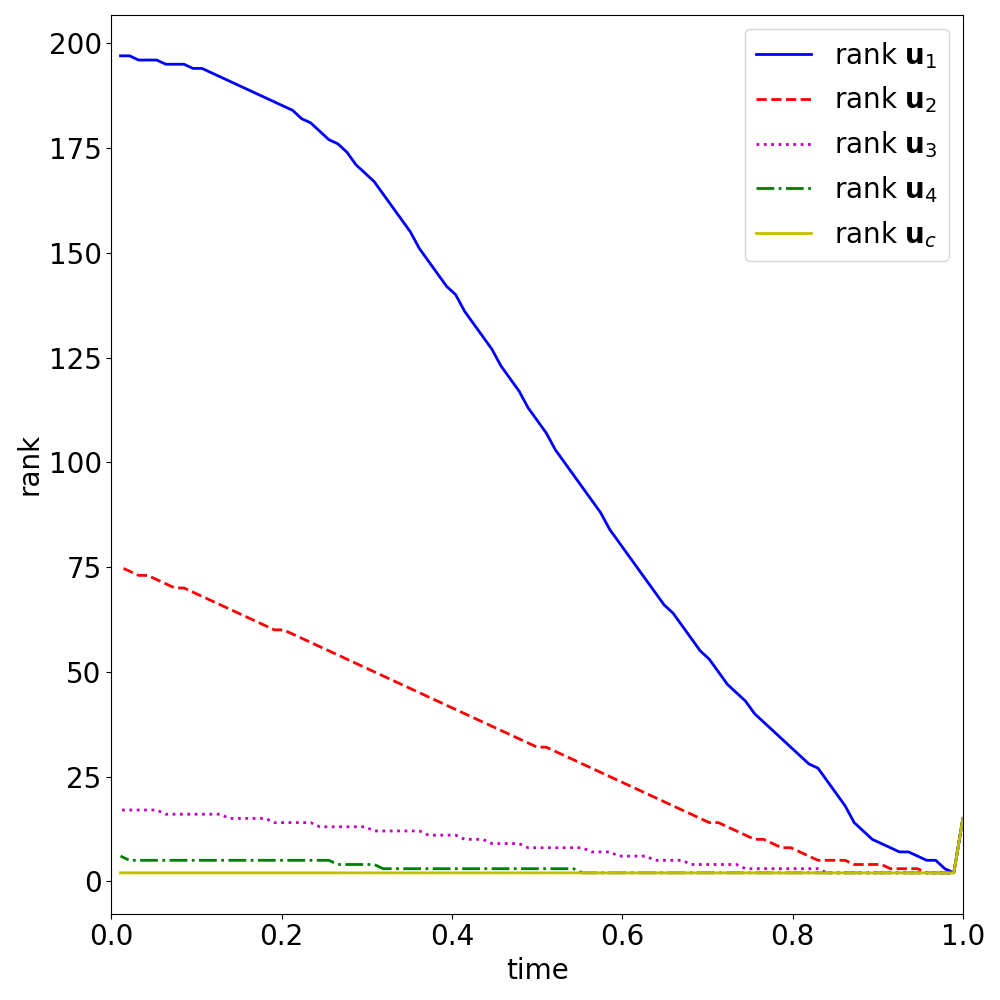}
  \caption{$\vartheta = 0.3, L = 4$}
  \label{fig:sub2}
\end{subfigure}
\caption{Rank evolution in energy for different $\vartheta = \bar\vartheta\cdot\Vert \mathbf{S} \Vert_F$.}
\label{fig:rankInTime}
\end{figure}
\subsection{Treatment planning for lung patient}
In the following we examine the application of the proposed method to a realistic 2D CT scan of a lung patient. The patient is radiated with an electron beam of $E_{\text{max}} = 21$ MeV. We model this beam as
\begin{align*}
    \psi_{\text{in}}(E,\mathbf x, \mathbf \Omega) = 10^5&\cdot\exp(-(\Omega_{1,*}-\Omega_1)^2/\sigma_{\Omega_1})\cdot\exp(-(E_{\text{max}}-E)^2/ \sigma_E)\\
    &\cdot\exp(-(x_*-x)^2/\sigma_x)\cdot\exp(-(y_*-y)^2/\sigma_y)\;,
\end{align*}
which is used as boundary condition for the uncollided particle flux. To determine a tissue density $\rho$ for given gray-scale values of the CT image, we assume a value of one, i.e., a white pixel to consist of bone material with density $\rho_{\text{bone}} = 1.85\text{ g}/\text{cm}^3$. The remaining tissue is scaled such that a pixel value of zero corresponds to a minimal density of $\rho_{\text{min}} = 0.05 \text{ g}/\text{cm}^3$. Air around the patient is filled with material, since this region does not impact the dose distribution. The chosen settings are the same as in Section~\ref{sec:linesource}. Since we are using a directed particle beam as boundary condition for the uncollided particles, the number of quadrature points $n_q$ reduces by over $59$ percent. The remaining parameters are:
\begin{center}
    \begin{tabular}{ | l | p{8cm} |}
    \hline
    $n_q = 396$ & number of quadrature points for uncollided flux \\
    $E_{\text{max}}=21$ & energy of beam in MeV \\
    $x_* = 7.25, y_* = 14.5$ & spatial mean of particle beam in cm\\
    $\Omega_{1,*} = 1$ & directional mean of particle beam \\
    $\sigma_{\Omega_1}^{-1}= 75$ & inverse directional beam variance \\
    $\sigma_x^{-1}=\sigma_y^{-1} = 20$ & inverse spatial beam variance \\
    $\sigma_E^{-1}=100$ & inverse energy variance \\
    \hline
    \end{tabular}
\end{center}
\begin{figure}[htp!]
    \centering
    \includegraphics[width=\linewidth]{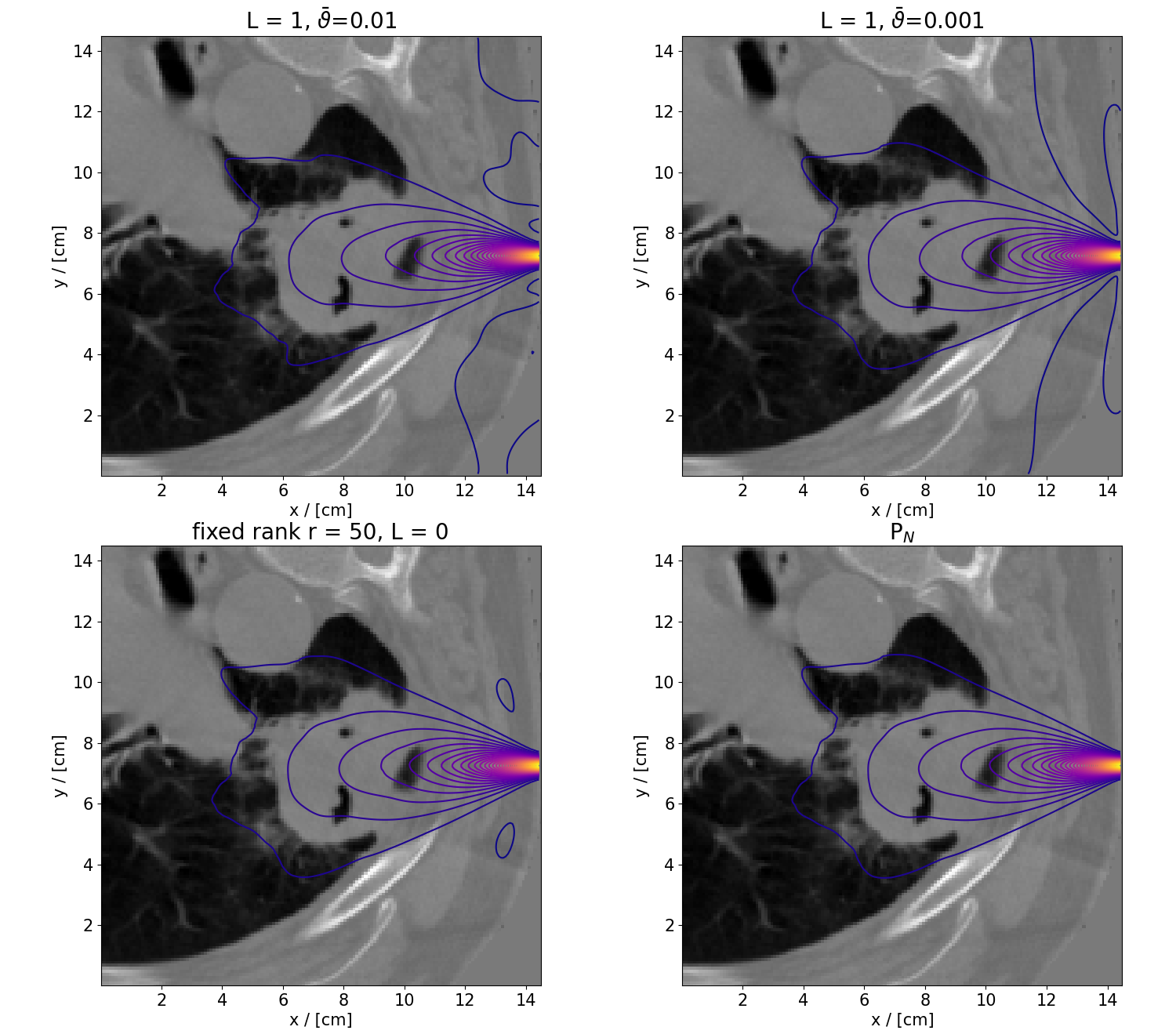}
    \caption{Dose distribution with different methods.}
    \label{fig:doseLung}
\end{figure}
For this setting we compute the full P$_N$ solution, the proposed dynamical low-rank method with a fixed rank of $50$ consisting only of collided and uncollided particles as well as the rank adaptive version with $L=1$ intermediate levels. Due to its reduced computational costs, the DLRA methods show a significantly reduced runtime. While the full P$_N$ method runs for 47329 seconds, the DLRA methods have a runtime of 4306, 4635 and 5392 seconds respectively. The resulting dose distribution can be found in Figure~\ref{fig:doseLung}. All considered variations of the proposed method are able to capture the effect of heterogeneities in the patient density and agree very well with the P$_N$ solution in the relevant dose areas. The efficiency of the method concerning both time and memory makes it feasible for practical applications. This includes the generation of optimal treatment plans with gradient-based optimization methods. It is observed that choosing a low refinement tolerance $\vartheta = \bar{\vartheta}\cdot\Vert \mathbf{S}\Vert_F$ with $\bar{\vartheta} = 0.01$ leads to a slight difference to the full solution for the smallest isoline. Note that isolines appearing below and above the main beam on the right are artifacts by dose values close to zero and are not of interest.
\begin{figure}[htp!]
\centering
\begin{subfigure}{.5\textwidth}
  \centering
  \includegraphics[width=\linewidth]{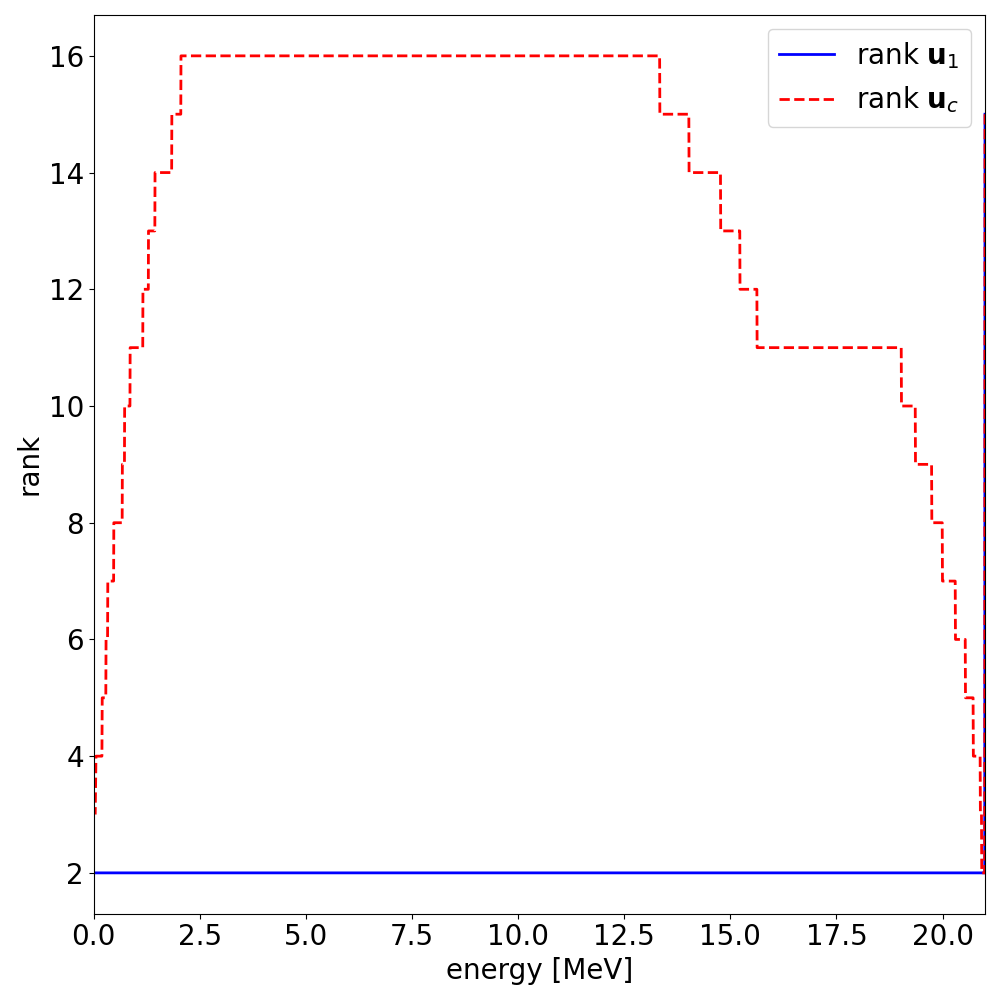}
  \caption{$\bar\vartheta = 0.01$}
  \label{fig:sub1}
\end{subfigure}%
\begin{subfigure}{.5\textwidth}
  \centering
  \includegraphics[width=\linewidth]{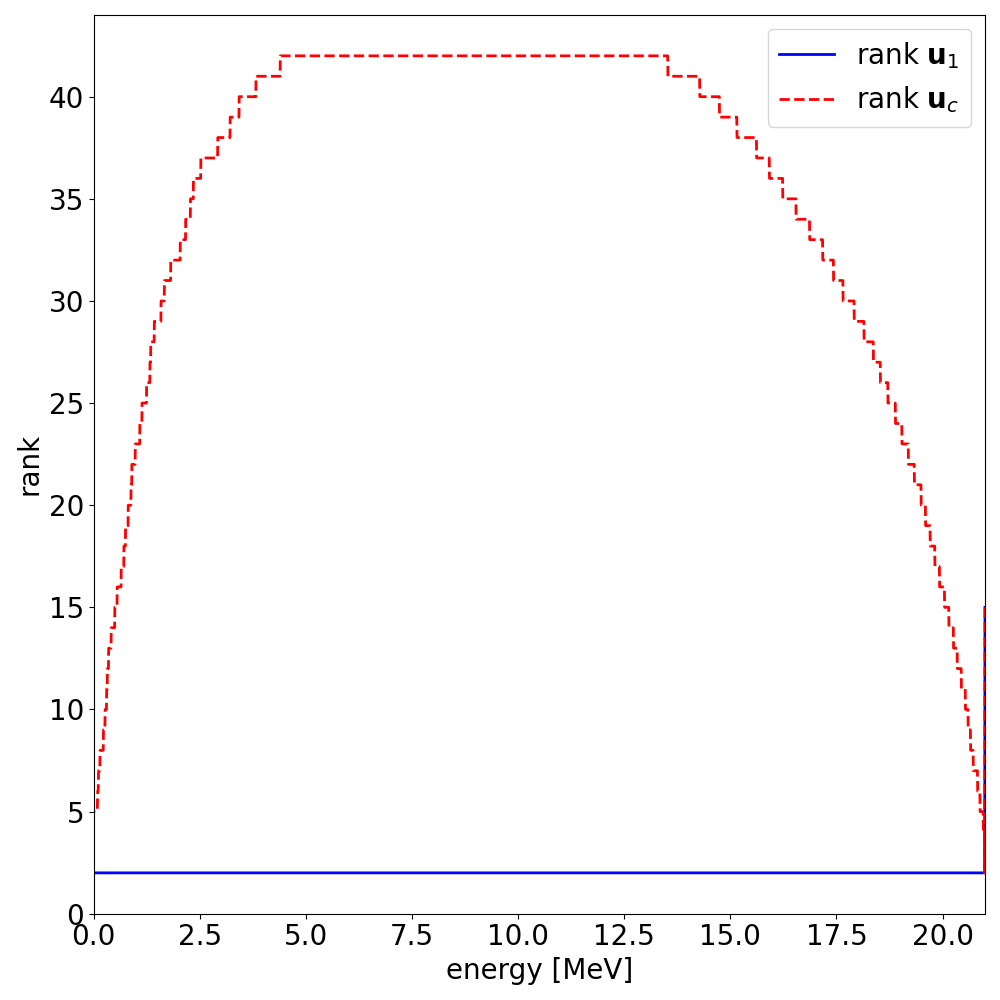}
  \caption{$\bar\vartheta = 0.001$}
  \label{fig:sub2}
\end{subfigure}
\caption{Rank evolution in energy for different $\vartheta = \bar\vartheta\cdot\Vert \mathbf{S} \Vert_F$.}
\label{fig:rankInEnergy}
\end{figure}
The corresponding ranks at different energies for $\bar{\vartheta}\in\{0.001,0.01\}$ are depicted in Figure~\ref{fig:rankInEnergy}. It is observed that the rank of collided particles remains small at low and high energies. At intermediate energies, the rank reaches its maximum. Particles that have collided once can be described with a small rank throughout the simulation. Note that particles which have collided once are not present for energies below $15$ MeV. The reason for this is that particles directly enter the patient tissue and are therefore directly subject to scattering.

The first four dominant spatial modes at the lowest energy are depicted in Figure~\ref{fig:spatialModesLung} and the first four dominant directional modes are shown in Figure~\ref{fig:directionalModesLung}. These modes have been computed by an SVD of the coefficient matrix $\mathbf{S} = \mathbf{U}\mathbf{D}\mathbf{V}^T$. We then plot the first four columns of $\mathbf{X}\mathbf{U}$ and $\mathbf{W}\mathbf{V}$. It is observed that the directional basis carries the information that particles are predominantly travelling into the $x$-direction, i.e. into the direction of the particle beam. The spatial basis encodes that particles with low energies are situated at the left of the CT scan and can mostly be found in high-density tissue. Note that due to the low energy of these particles, no significant contribution to the overall dose distribution is observed.
\begin{figure}[htp!]
    \centering
    \includegraphics[width=\linewidth]{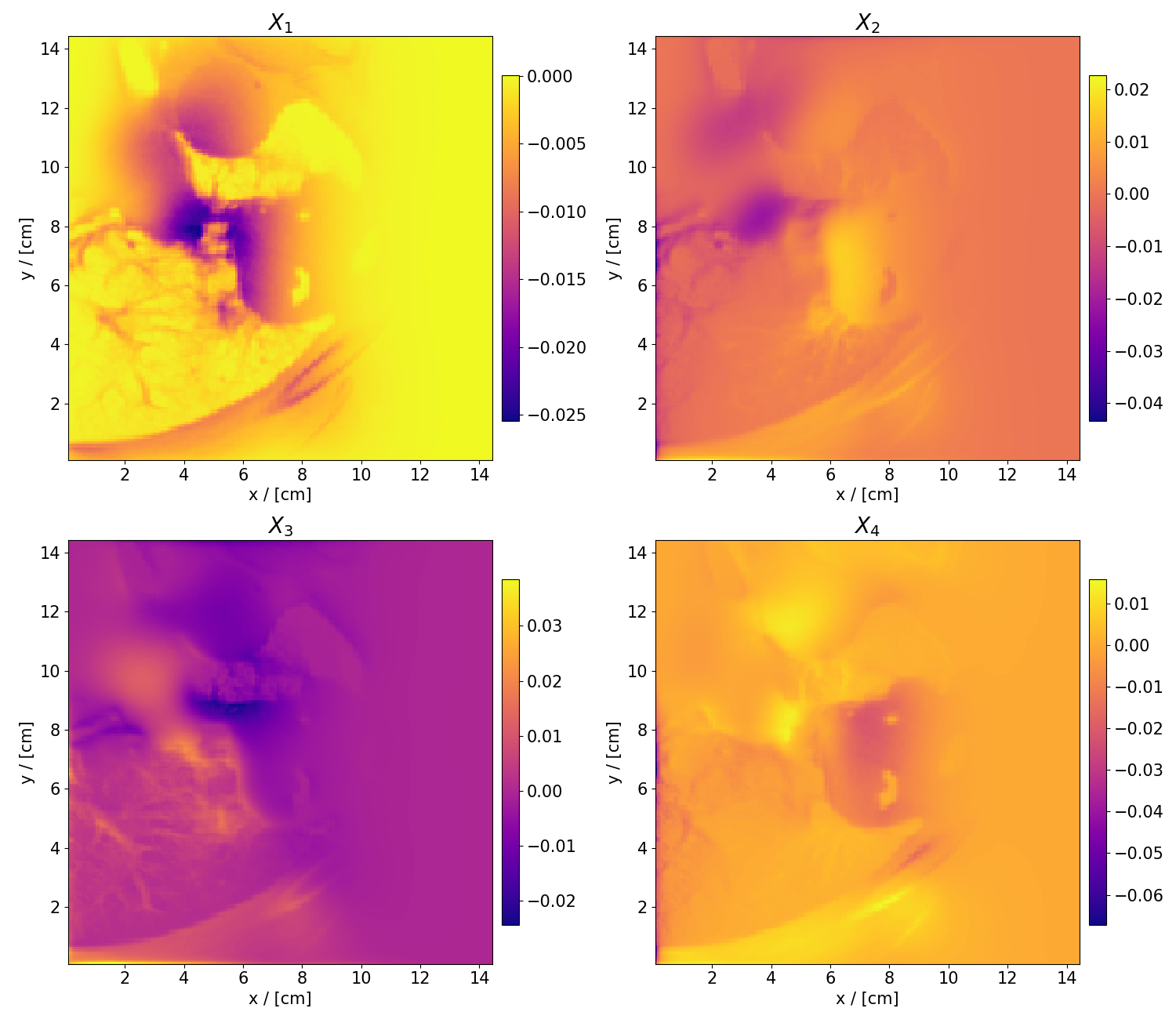}
    \caption{First four dominant spatial modes with fixed rank integrator.}
    \label{fig:spatialModesLung}
\end{figure}
\begin{figure}[htp!]
    \centering
    \includegraphics[width=\linewidth]{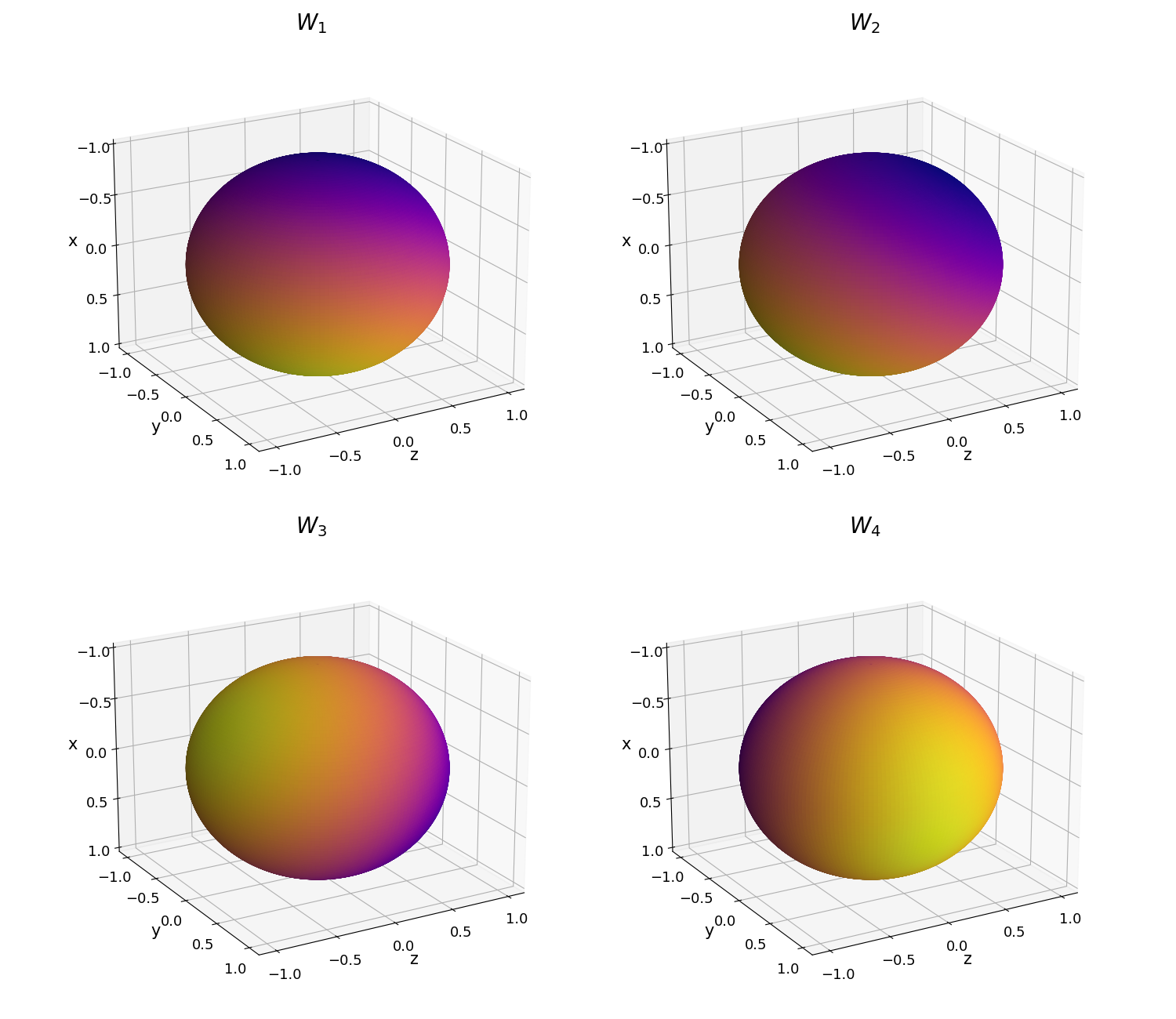}
    \caption{First four dominant directional modes with fixed rank integrator.}
    \label{fig:directionalModesLung}
\end{figure}


\section{Outlook}
In future work, we aim at using the proposed forward method to facilitate optimization and uncertainty quantification in radiation therapy. To improve the understanding of dynamical low-rank approximation in this new field of application, further validations, especially in 3D geometries and against the current standard pencil beam or MC algorithms are necessary. However, our results promise an approach which is efficient enough for a use in dose optimization, while still taking into account all relevant physical interactions.

\section*{Acknowledgments}
	Jonas Kusch has been funded by the Deutsche Forschungsgemeinschaft (DFG, German Research Foundation) --- Project-ID 258734477 --- SFB 1173. Pia Stammer is supported by the Helmholtz Association under the joint research school HIDSS4Health -- Helmholtz Information and Data Science School for Health.
	
	\bibliographystyle{abbrv}
	\bibliography{main} 
	
\end{document}